\documentclass{amsproc}
\usepackage{amsrefs,amssymb,enumitem,myproof}

\usepackage[colorlinks,bookmarks=false,linkcolor=blue,linktocpage]{hyperref}
\usepackage[capitalize,nameinlink,noabbrev]{cleveref}
\crefformat{equation}{\textup{#2(#1)#3}}
\crefname{page}{page}{pages}
\crefname{claim}{Claim}{Claims}
\crefname{enumi}{step}{steps}
\crefname{enumii}{step}{steps}
\crefname{enumii}{step}{steps}

\newtheorem{theo}{Theorem}[section]
\newtheorem{theorem}[theo]{Theorem}
\newtheorem{lemma}[theo]{Lemma}

\newtheorem{prop}[theo]{Proposition}

\newtheorem{claim0}{Claim}
\newtheorem{claim}[claim0]{Claim}

\theoremstyle{remark}
\newtheorem{example}[theo]{Example}
\newtheorem{alg}[theo]{Algorithm}
\newtheorem{remark}[theo]{Remark}

\newenvironment{ex}{\begin{example}}{\exend\end{example}}
\newenvironment{rem}{\begin{remark}}{\exend\end{remark}}

\newcommand{\exend}{\hfill\ensuremath{\triangleleft}}
\newcommand{\upto}{,\ldots,}

\newcommand{\Spec}{\operatorname{Spec}}
\newcommand{\GL}{\operatorname{GL}}
\newcommand{\ch}{\operatorname{char}}
\newcommand{\im}{\operatorname{im}}
\newcommand{\LM}{\operatorname{LM}}
\newcommand{\NF}{\operatorname{NF}}
\newcommand{\Tr}{\operatorname{Tr}}
\newcommand{\Ext}{\operatorname{Ext}}
\newcommand{\codim}{\operatorname{codim}}
\newcommand{\Var}{\mathcal V}
\newcommand{\Id}{\mathcal I}
\renewcommand{\phi}{\varphi}
\newcommand{\RR}{{\mathbb R}}
\newcommand{\CC}{{\mathbb C}}

\newcommand{\ZZ}{{\mathbb Z}}

\newcommand{\FF}{{\mathbb F}}

\newcommand{\df}[1]{{\bf #1\/}}
\newcommand{\R}[1][K]{#1[\mathbf x]}%
\newcommand{\I}{H}%

\begin{document}

\title[Invariant Theory: a Third Lease of Life]{Invariant Theory: a
  Third Lease of Life\footnote{\bf This article is intended for the
    volume ``Combinatorial, Computational, and Applied Algebraic
    Geometry: A Tribute to Bernd Sturmfels'' edited by Serkan Hosten,
    Diane Maclagan, and Frank Sottile.}}

\author{Gregor Kemper}%
\address{Technische Universi\"at M\"unchen, Department of Mathematics,
  Boltzmannstr. 3, 85748 Garching, Germany}%
\email{kemper@ma.tum.de}
\date{}


\begin{abstract}
  In 1993, just about a century after the epoch of Classical Invariant
  Theory and almost 30 years after Mumford's seminal book on Geometric
  Invariant Theory, Bernd Sturmfels approached the subject from a new,
  algorithmic perspective in
  his book on Algorithms in Invariant Theory. This article aims to
  highlight some of the developments that followed the book. Inspired
  by Bernd's style of teaching mathematics, the goal is neither
  comprehensiveness nor maximal generality, but to emphasize the main
  ideas and to convey the beauty of the subject. The article is
  intended as an invitation to invariant theory, and in particular to
  Bernd's work and the ideas inspired by it.
\end{abstract}

\maketitle

\section*{Introduction}

When Bernd Sturmfels's book on Algorithms in Invariant
Theory~\cite{stu} appeared in 1993, I immediately devoured it, since
at the time I was busy developing a package for computing invariant
rings in the Maple computer algebra system. A major part of my
Ph.D. thesis was concerned with computational invariant theory, and
Bernd's book served as an indication that this would be the part with
the best chances of getting people interested. As it turned out, the
Bernd's book gave the subject a significant boost, especially by its
overarching theme of wedding invariant theory to Gr\"obner basis
methods. This is the idea behind the ``third lease of life'' in the
title of this article, with the previous ``leases'' thought of as
Classical Invariant Theory and Geometric Invariant Theory. Bernd's
book came at a time when there was a confluence of new developments in
computational invariant theory and in modular invariant theory, with
Harm Derksen's algorithm emerging shortly after the book's
publication.

The book also got me in contact with Bernd, who later invited me as a
contributor to a course he gave on the book in Eindhoven---my first
international mathematical gig. Years after that, in 1999, he invited
Harm Derksen and me to Berkeley at a time when the two of us were
embarking on a joint book project in computational invariant
theory~\cite{Derksen:Kemper}. Let me quote from the preface: ``{\em
  \ldots it was an invitation by Bernd Sturmfels to spend two weeks
  together in Berkeley that really got us started on this book
  project. We thank Bernd for his strong encouragement and very
  helpful advice. During the stay at Berkeley, we started outlining
  the book, making decisions about notation, etc.}''

After writing his invariant theory book, Bernd, as he so often does,
rather quickly moved on to other topics. But invariant theory, having
received a new impetus from his book, also moved on, and it is the
purpose of this article to cast some light on what has happened since
1993. However, this is not meant to be a typical survey article, since
it will not aim for comprehensiveness and since it will contain quite
a lot of proofs. Readers should not only learn what is true, but also,
when space permits, why it is true. In this way, I will try to mimic
what I understand to be Bernd's style of teaching mathematics. For the
selection of material to be presented here, one criterion is that it
should be at least loosely related to computation. Another criterion
is that the results and the proofs, when they are included, should be
nice, or even beautiful and elegant. So there is definitely some
cherry-picking going on. (And, inevitably, my personal taste will
favor some cherries over others.) To make the material as accessible
as possible and to keep things simple for readers, attaining maximal
generality will not be a paramount goal. This, I believe, is also in
keeping with Bernd's style of teaching.

Keeping things simple for readers also means to avoid hopping between
different situations and changing notation. In this article, we will
always consider a group $G$ acting on a polynomial ring
$K[x_1 \upto x_n] =: \R$ over a field by linear transformations of the
variables~$x_i$. The first and foremost object of our interest will be
the invariant ring $\R^G$. Here is another convention used throughout:
when I say ``Bernd's book'', I will always be referring to his
book~\cite{stu} on Algorithms in Invariant Theory.

The first three sections will look at the case where $G$ is finite,
distinguishing between the \df{nonmodular} case (where the
characteristic of $K$ does not divide $|G|$, including $\ch(K) = 0$)
and the \df{modular} case. Specifically, \cref{sNonmodular} is devoted
to nonmodular invariant theory and will present Noether's degree bound
and King's algorithm for computing $\R^G$. In \cref{sModular} on
modular invariant theory of finite groups, we address Symonds's degree
bound (without proof) and look at results about the Cohen-Macaulay
property (giving a full proof for one of them). The third section
deals with separating invariants. We show that they always satisfy
Noether's degree bound, and give a self-contained proof of a result by
Emilie Dufresne about separating invariants and reflection groups.

The last three sections are about invariants of infinite groups. They
are significantly shorter than the first three sections, since there
are fewer new results to report. But this does not mean that they are
any less important. Maybe even on the contrary, as exemplified by
Derksen's algorithm, which is the topic of \cref{sReductive}. In the
fifth section we show how invariant fields of potentially nonreductive
groups can be computed using the same ideal that occurs prominently in
Derksen's algorithm, but in a very different way. The section also
explains how to compute invariant rings of reductive groups in
positive characteristic. For this, the stepping stone is the
computation of separating invariants, which are the topic of
\cref{sSepInf}. Here we present a proof of the (folklore) statement
that there is always a separating set of size $\le 2 n + 1$,
independently of the group.

\paragraph{\bf Acknowledgments} I would like to thank Serkan Hosten,
Diane Maclagan, and Frank Sottile for organizing this volume and, of
course, for inviting me to make a contribution. It was a pleasure to
do so. Also many thanks to the anonymous reviewers for their careful
reading of the manuscript and for numerous helpful suggestions,
corrections and comments.


\section{Nonmodular invariants of finite groups}
\label{sNonmodular}

In nonmodular invariant theory of finite groups (more generally, in
invariant theory of linearly reductive groups) the supreme ruler is
the \df{Reynolds operator}
\[
  \mathcal R: \ \R \to \R^G, \ f \mapsto \frac{1}{|G|} \sum_{\sigma
    \in G} \sigma(f),
\]
which averages over the group.
It makes an appearance in the proof of the following proposition,
which in some form was already known to Hilbert and Noether. (More
than a hundred years later, I learned it from Bernd's book.) The
proposition tells us how the ideal
\[
\I := \R \cdot \R_+^G
\]
generated by all nonconstant homogeneous invariants plays a central
role. (Here $\R_+^G$ stands for the set of all invariants with
constant coefficient equal to zero, and the product $\R \cdot \R_+^G$
consists of all finite sums of products of an element from $\R$ and
one from $\R_+^G$.) $H$ has come to be known as the \df{Hilbert
  ideal}, and it is an ideal in $\R$.

\begin{prop} \label{pHilbertIdeal}%
  For nonconstant homogeneous invariants $f_1 \upto f_m \in \R^G$, it
  is equivalent that they generate the Hilbert ideal $\I$ (as an ideal) 
  and that they generate the invariant ring $\R^G$ (as an algebra).
\end{prop}

Another equivalent condition is that (the classes of) the~$f_i$
generate the quotient space $\I/(\R_+ \cdot \I)$ as a vector space over
$K$. This fact is not needed here, and its proof is left as an
exercise.

\begin{proof}[Proof of \cref{pHilbertIdeal}]
  If the~$f_i$ generate $\R^G$, then every invariant~$h \in \R_+^G$
  can be written as $h = F(f_1 \upto f_m)$ with $F$ a polynomial
  in~$m$ variables having zero constant coefficient. Since every
  product of powers of some of the~$f_i$ lies in the $\R$-ideal
  $(f_1 \upto f_m)_{\R}$ generated by the~$f_i$, so does~$h$. We
  conclude that $\R_+^G \subseteq (f_1 \upto f_m)_{\R}$ and thus 
  $\I \subseteq (f_1 \upto f_m)_{\R}$. The reverse inclusion follows
  from $f_i \in \I$.

  For the converse, assume that $\I = (f_1 \upto f_m)_{\R}$ and
  let~$f$ be a homogeneous invariant of degree~$d > 0$. Then 
  $f \in \I$, so $f = \sum_{i=1}^m g_i f_i$ with $g_i \in \R$. 
  Applying the Reynolds operator yields
  \[
    f = \mathcal R(f) = \sum_{i=1}^m \mathcal R(g_i) f_i.
  \]
  By considering the homogeneous part of degree~$d$ of this equation,
  we may assume that all summands are homogeneous of degree~$d$, so
  the $\mathcal R(g_i)$ have degree $< d$. So by induction on~$d$ we
  may assume that they lie in the algebra $K[f_1 \upto f_m]$ generated
  by the~$f_i$. Therefore the same is true for~$f$.
\end{proof}

Together with Hilbert's basis theorem, which by historical accounts
was proved for this very purpose, the proposition already shows that
if there is a Reynolds operator, then the invariant ring is finitely
generated. It also follows that any degree bound for generators of the
Hilbert ideal is automatically a degree bound for the invariant ring.

This brings us to Noether's degree bound, proved by Emmy
Noether~\cite{noe:a} in 1916, which says that the invariant ring is
generated by homogeneous invariants of degree at most the group order
$|G|$. Actually, she gave two proofs. However, the first proof
requires that $|G|!$ is invertible in $K$, so the characteristic must
be~$0$ or bigger than $|G|$, and the second one, which can be found in
Bernd's book~\cite[Theorem~2.1.4]{stu}, only works in
characteristic~$0$. A proof that extends to the full nonmodular
case---people were speaking of the ``Noether gap''---remained elusive
for quite some time, until it was finally found independently by
Fleischmann~\cite{fleisch:c} and Fogarty~\cite{Fogarty:2001}. We give
a version that has gone through various stages of simplification. In
fact, Noether's bound is an almost immediate consequence of the
following lemma, which makes a surprising observation and is proved by
a nice, elementary trick.

\begin{lemma} \label{lHilbertIdeal}%
  The Hilbert ideal $\I$ contains every monomial in the~$x_i$ of
  degree~$|G|$.
\end{lemma}

\begin{proof}
  We will show a bit more: that $\I$ contains every product of~$|G|$
  (not necessarily distinct) polynomials of $\R_+$. Such polynomials
  may as well be indexed by the elements of $G$, so we claim that any
  product $\prod_{\sigma \in G} f_\sigma$, with $f_\sigma \in \R_+$,
  lies in $\I$. For every $\tau \in G$ we have
  \[
    \prod_{\sigma \in G} \bigl(f_\sigma - (\tau\sigma)(f_\sigma)\bigr)
    = 0.
  \]
  Fully expanding the product and summing over all $\tau \in G$ yields
  \[
    \sum_{M \subseteq G} (-1)^{|M|} \left(\prod_{\sigma \in G
        \smash{\setminus} M} f_\sigma\right) \cdot \left(\sum_{\tau
        \in G}\prod_{\sigma \in M} (\tau \sigma) (f_\sigma)\right) =
    0.
  \]
  The summand for $M = \emptyset$ is
  $|G| \cdot \prod_{\sigma \in G} f_\sigma$, and all other summands
  lie in $\I$. So dividing by $|G|$ proves the claim.
\end{proof}

\begin{theorem}[Noether's degree bound] \label{tNoether}%
  If $|G|$ is invertible in $K$, then the invariant ring $\R^G$ is
  generated as an algebra by homogeneous invariants of degree $\le |G|$.
\end{theorem}

\begin{proof}
  Let $I \subseteq \R$ be the ideal generated by all nonconstant
  homogeneous invariants of degree $\le |G|$. By \cref{lHilbertIdeal},
  $I$ contains every monomial of degree~$|G|$, and therefore also
  every monomial of degree $\ge |G|$. So $I$ contains every
  homogeneous polynomial of degree $\ge |G|$. This implies that $I$
  contains every nonconstant homogeneous invariant. Thus the Hilbert
  ideal $H$ is generated in degree $\le |G|$, and by
  \cref{pHilbertIdeal} the same is true for $\R^G$.
\end{proof}

In summary, we have presented a self-contained proof of Noether's
bound that, after removal of comments and other luxuries, fits on a
single page and also closes the Noether gap. \\

With a degree bound, we automatically have an algorithm for computing
generators of the invariant ring: simply apply the Reynolds operator
to all monomials of degree $\le |G|$. Of course, this is extremely
inefficient. In Bernd's book two algorithms are given for computing
nonmodular invariant rings of finite groups. The first one
\cite[Algorithm~2.2.5]{stu} is driven by knowledge of the Hilbert
series of the invariant ring, and the second
\cite[Algorithm~2.5.8]{stu} computes a so-called Hironaka
decomposition. This subdivides the generating invariants into primary
invariants, which are algebraically independent, and secondary
invariants, which generate the invariant ring as a module over the
subalgebra generated by the primary invariants. The second algorithm
from Bernd's book has gone through various optimizations (for example
\cite{kem:g} improves the finding of primary invariants), but has
remained the state of the art for about 20 years until Simon King
\cite{King:2013} proposed a new one in 2013. Although more elementary
than Sturmfels' algorithms, King's algorithm is very clever and
therefore makes good material for this article. Unlike the algorithms
given by Sturmfels, King's algorithm requires only {\em truncated}
Gr\"obner bases (see, for example,
\cite[Section~4.5.B]{Kreuzer:Robbiano:2005}). For this reason,
intuition says that it should be more effective and, more importantly,
experience confirms this. That is why King's algorithm is now the
standard in Magma \cite{magma} and Macaulay2 \cite{FGGHMN:2020},
qualifying it to be called today's state of the art. At the time of
writing this, it is also becoming part of the brand new computer
algebra system OSCAR~[\url{https://oscar.computeralgebra.de/}].  So
here it is:

\begin{alg}[King's algorithm~\cite{King:2013}] \label{aKing} \mbox{}
  \begin{description}
  \item[Input] A finite group $G$ acting linearly on the
    variables~$x_i$, with $|G|$ not a multiple of $\ch(K)$.
  \item[Output] A minimal generating set of the invariant ring $\R^G$
    as an algebra.
  \end{description}
  \begin{enumerate}[label=(\arabic*)]
  \item \label{aKing1} Set $S := \emptyset$. Choose a monomial
    ordering on $\R$.
  \item \label{aKing2} For $d = 1,2, \ldots$ execute the following
    steps:
    \begin{enumerate}[label=(\arabic*)]
      \setcounter{enumii}{\arabic{enumi}}
    \item \label{aKing3} Let $\mathcal G$ be a $d$-truncated Gr\"obner
      basis of the ideal $(S) \subseteq \R$ generated by $S$. This can
      be computed by Buchberger's algorithm, but considering only
      s-polynomials of degree~$\le d$.
    \item \label{aKing4} Let $M$ be the set of all monomials in $\R$
      of degree~$d$ that are not divisible by any leading monomial
      $\LM(g)$, $g \in \mathcal G$.
    \item \label{aKing5} For $t \in M$ execute the following two
      steps:
      \begin{enumerate}[label=(\arabic*)]
        \setcounter{enumiii}{\arabic{enumii}}
      \item \label{aKing6} Compute $f := \mathcal R(t)$, with
        $\mathcal R$ the Reynolds operator, and
        $h := \NF_{\mathcal G}(f)$, the {\em normal form} of~$f$ with
        respect to $\mathcal G$. By this we mean the unique
        polynomial~$h$ with $f - h \in (\mathcal G)$ such that no
        monomial occurring in~$h$ is divisible by any $\LM(g)$,
        $g \in \mathcal G$.
      \item \label{aKing7} If $h \ne 0$, then adjoin~$f$ to $S$ and~$h$ to
        $\mathcal G$, and set $M := M \setminus \{\LM(h)\}$.
      \end{enumerate}
      \setcounter{enumii}{\arabic{enumiii}}
    \item \label{aKing8} If $M = \emptyset$, then terminate the
      algorithm and return $S$.
    \end{enumerate}
  \end{enumerate}
\end{alg}

I find it hard to understand immediately what is going on in the
algorithm and why it works, and readers might feel the same
way. Especially the consecutive application of the Reynolds operator
and the normal form in \cref{aKing6} and the business of removing
monomials from $M$ during the second loop may seem a bit obscure. To
me, the only way to understand the algorithm is through a proof of
correctness, which is presented here.

\begin{proof}[Proof of correctness of \cref{aKing}]
  The proof is subdivided into five claims.

  \begin{claim} \label{c1}%
    During the inner loop initiated by \cref{aKing5}, $\mathcal G$
    remains a $d$-truncated Gr\"obner basis of $(S)$ and $M$ remains
    the set of monomials of degree~$d$ not divisible by any $\LM(g)$,
    $g \in \mathcal G$.
  \end{claim}

  Indeed, since each~$h$ that is adjoined to $\mathcal G$ is in normal
  form with respect to the elements already in $\mathcal G$, no new
  s-polynomials of degree $\le d$ occur after adjoining~$h$, so the
  ``new'' $\mathcal G$ is again a $d$-truncated Gr\"obner basis. It
  also generates the same ideal as the ``new'' $S$. Moreover, since
  $\LM(h)$ has degree~$d$, removing it from $M$ amounts to removing
  all monomials divisible by it.

  \begin{claim} \label{c2}%
    After~$d$ passes through the outer loop initiated by
    \cref{aKing2}, all homogeneous invariants of degree $\le d$ lie in
    the algebra generated by $S$: $\R^G_{\le d} \subseteq K[S]$.
  \end{claim}

  The proof proceeds by induction on~$d$, so we assume
  $\R^G_{\le d-1} \subseteq K[S]$ at the start of the $d$th pass. In
  the following, let $S$, $\mathcal G$, and $M$ denote the sets as
  they are {\em after} the $d$th pass. Let~$\widetilde f$ be a
  homogeneous invariant of degree~$d$. Using \cref{c1}, we get
  $\widetilde f - \NF_{\mathcal G}(\widetilde f) \in (\mathcal G) =
  (S)$, so
  $\widetilde f - \NF_{\mathcal G}(\widetilde f) = \sum_{i=1}^r c_i
  f_i$ with $f_i \in S$ and $c_i \in \R$. Again by \cref{c1},
  $\NF_{\mathcal G}(\widetilde f)$ is a linear combination of
  monomials~$t_1 \upto t_s$ from $M$, so
  \[
    \widetilde{f} = \sum_{i=1}^r c_i f_i + \sum_{i=1}^s a_i t_i
  \]
  with~$a_i \in K$. Applying the Reynolds operator yields
  \[
    \widetilde{f} = \sum_{i=1}^r \mathcal R(c_i) f_i + \sum_{i=1}^s
    a_i \mathcal R(t_i).
  \]
  By considering only the degree-$d$ part we may assume the
  $\mathcal R(c_i)$ to be homogeneous, so their degrees are $< d$, and
  thus $\mathcal R(c_i) \in K[S]$ by induction. So \cref{c2} is proved
  if we can show $\mathcal R(t_i) \in K[S]$ for all~$i$. Since
  the~$t_i$ lie in the set $M$ as it is after the $d$-th pass through
  the outer loop, they have {\em not} been removed in
  \cref{aKing7}. So for $t = t_i$ and $f = \mathcal R(t)$ we have
  $h = \NF_{\mathcal G}(f) = 0$. Hence it is enough to prove

  \begin{claim} \label{c3}%
    For~$f$ and~$h$ in \cref{aKing6}, we have the equivalence
    \[
      h = 0 \quad \Longleftrightarrow \quad f \in K[S],
    \]
    where $S$ denotes the set as it is in \cref{aKing6}.
  \end{claim}

  In fact, \cref{c1} shows that $h = 0$ is equivalent to $f \in
  (S)$. If that happens we have $f = \sum_{i=1}^r q_i f_i$ with
  $f_i \in S$ and $q_i \in \R$. We can again apply the Reynolds
  operator, assume homogeneity and use induction to show that the
  $q_i$ may be taken from $K[S]$, so $f \in K[S]$. Conversely, if
  $f \in K[S]$, then clearly $f \in (S)$ and hence $h = 0$. So
  \cref{c2,c3} are proved.
  Let us also note that the equivalence in \cref{c3} means that an
  invariant~$f$ is adjoined to $S$ only if this enlarges $K[S]$. This
  implies that throughout the algorithm, $S$ is a {\em minimal}
  generating set. Now we aim to show that the algorithm terminates,
  and that upon termination, $S$ generates $\R^G$.

  \begin{claim} \label{c4}%
    The algorithm terminates after $|G|$ passes through the outer loop
    or earlier.
  \end{claim}

  Indeed, \cref{c2} implies that after~$d$ passes, $(S)$ contains all
  homogeneous elements of the Hilbert ideal that have degree $\le
  d$. So if $d = |G|$, \cref{lHilbertIdeal} tells us that $(S)$
  contains all monomials of degree~$d$. Now it follows from \cref{c1}
  that all these monomials are divisible by $\LM(g)$ for some
  $g \in \mathcal G$. So using \cref{c1} again, we see that after the
  $|G|$th pass, $M$ will be empty, triggering termination.

  \begin{claim} \label{c5}%
    Upon termination, $\R^G = K[S]$.
  \end{claim}

  Indeed, if every monomial of degree~$d$ is divisible by some
  $\LM(g)$, then the same is true for every monomial of degree $>
  d$. This remains true if the truncated Gr\"obner basis is continued
  to a higher degree. So if the algorithm were left to keep running,
  it would henceforth always produce $M = \emptyset$ in \cref{aKing4}
  and would thus never adjoin a new invariant to $S$. So by \cref{c2},
  homogeneous invariants of any degree lie in $K[S]$, which proves the
  claim.
\end{proof}


Readers who take a look at the paper~\cite{King:2013} of King will
notice some differences. In fact, our \cref{aKing} is a modification
the original one. For one thing, it is simpler, and for another, the
original algorithm does in fact require the computation of a full
Gr\"obner basis at some point, so the previous statement that only
truncated Gr\"obner bases are used really applied to the modification
given here. Specifically, in the original algorithm the computation of
a full Gr\"obner basis of the ideal $(S)$ is triggered if for some~$d$
no invariants of that degree have been added. If $(S)$ has
dimension~$0$ at that point, the Gr\"obner basis is used to determine
a bound for termination. If it does not, a full Gr\"obner basis is
computed again at a later point. The variant presented here as
\cref{aKing} does not need such heuristics and terminates at the
same~$d$ as the original one, or earlier.

\begin{ex} \label{eD8}%
  Let us run \cref{aKing} on a group taken from Bernd's
  book~\cite[Proposition~2.2.10]{stu}. This is the symmetry group
  $G = D_8$ of a regular octagon, generated by a reflection
  $\tau = \left(\begin{smallmatrix} 1 & 0 \\ 0 &
      -1\end{smallmatrix}\right)$ and a rotation
  $\sigma = \frac{1}{\sqrt{2}}\left(\begin{smallmatrix} 1 & -1 \\ 1 &
      1\end{smallmatrix}\right)$ by an angle of $45^\circ$. $G$ has
  order~16, and plays an interesting role in coding theory, as
  explained in Bernd's book.

  Notice that $G$ contains the scalar matrix
  $\sigma^4 = \left(\begin{smallmatrix} -1 & 0 \\ 0 &
      -1\end{smallmatrix}\right)$, which has the consequence that no
  homogeneous invariant of odd degree exists. This means that for
  odd~$d$ the invariants produced by the algorithm will all be~$0$, so
  no additions to the generating set $S$ are made and we might as well
  skip those~$d$. We write~$x$ and~$y$ for the variables, and choose a
  monomial ordering with $x > y$. One further comment before we delve
  into the workings of the algorithm: in the inner loop over the
  monomials $t \in M$, it is clever to start with the least monomial
  because that gives a better chance that in \cref{aKing7} a monomial
  will be taken out of $M$ that has not been treated yet. Now here is
  what the algorithm does, rendered in telegraphic style, with the
  computations not shown in detail but left to the computer.
  \begin{description}
  \item[$\mathbf{d = 2}$] As initialized,
    $S = \mathcal G = \emptyset$, so $M := \{x^2,x y,y^2\}$.
    \begin{description}
    \item[$\mathbf{t = y^2}$]
      $f := \mathcal R(t) = \frac{1}{2} (x^2 + y^2) =: f_2$, $h :=
      f$. So we update $S := \{f_2\}$, $\mathcal G := \{f_2\}$, and
      $M := \{x y,y^2\}$.
    \item[$\mathbf{t = x y}$] $f := \mathcal R(t) = 0$.
    \end{description}
  \item[$\mathbf{d = 4}$] As before, $S = \mathcal G = \{f_2\}$ with
    leading monomial~$x^2$, so $M := \{x y^3,y^4\}$.
    \begin{description}
    \item[$\mathbf{t = y^4}$] $f := \mathcal R(t)$ is divisible
      by~$f_2$, so has normal form $h := 0$.
    \item[$\mathbf{t = x y^3}$] $f := \mathcal R(t) = 0$.
    \end{description}
  \item[$\mathbf{d = 6}$] As before, $S = \mathcal G = \{f_2\}$, so
    $M := \{x y^5,y^6\}$.
    \begin{description}
    \item[$\mathbf{t = y^6}$] $f := \mathcal R(t)$ is divisible
      by~$f_2$, so has normal form $h := 0$.
    \item[$\mathbf{t = x y^5}$] $f := \mathcal R(t) = 0$.
    \end{description}
  \item[$\mathbf{d = 8}$] As before, $S = \mathcal G = \{f_2\}$, so
    $M := \{x y^7,y^8\}$.
    \begin{description}
    \item[$\mathbf{t = y^8}$]
      $f := \mathcal R(t) = \frac{1}{32}(9x^8 + 28 x^6 y^2 + 70 x^4
      y^4 + 28 x^2 y^6 + 9 y^8) =: f_8$ has normal form $h := y^8$.
      So we update $S := \{f_2,f_8\}$, $\mathcal G := \{f_2,y^8\}$.
    \item[$\mathbf{t = x y^7}$] $f := \mathcal R(t) = 0$.
    \end{description}
  \item[$\mathbf{d = 9}$] (Even though we have until now skipped odd
    degrees, we do look at $d = 9$ since the computer, oblivious of
    our above argument about odd degrees, would do so, too.)
    $\mathcal G = \{f_2,y^8\}$ is a Gr\"obner basis, and therefore
    also a $9$-trunceated Gr\"obner basis, with leading
    monomials~$x^2$ and~$y^8$, so now $M := \emptyset$. This means
    that the algorithm terminates here.
  \end{description}
  In this example, the algorithm terminates at $d = 9$, one degree
  after the last invariant has been added, but before reaching
  $d = |G|$. It has found generating invariants~$f_2$ and~$f_8$, which
  may be replaced by
  \[
    g_2 = 2 f_2 = x^2 + y^2 \quad \text{and} \quad g_8 := 18 f_2^4 - 4
    f_8 = x^2 y^2 (x^2 - y^2)^2
  \]
  These are exactly the generating invariants given in Bernd's
  book. The invariant $g_8$ happens to be~$4$ times the product over
  the orbit of~$x$.
\end{ex}

\section{Modular invariants of finite groups} \label{sModular}

In this section we assume that the characteristic $p = \ch(K)$ divides
the group order $|G|$. Thus we no longer have a Reynolds
operator. So chaos is to be expected, and indeed many pleasant
features of nonmodular invariant theory break down.

The first victim is Noether's degree bound. Perhaps the easiest
example where it fails is the cyclic group $C_2$ of order~$2$ acting
on the polynomial ring \linebreak $\FF_2[x_1,x_2,x_3,y_1,y_2,y_3]$ by
interchanging the $x_i$ and $y_i$. An explicit calculation showing
that the invariant ring is not generated in degrees~$\le 2$ is given
in \cite[Example~3.3.1]{Derksen:Kemper:2015}, so we direct interested
readers there. Worse, Richman~\cite{rich:b} showed that in the modular
case there cannot exist any degree bound that only depends on
$|G|$. Reasonable degree bounds for the modular case were long elusive
(an unreasonably large bound appeared
in~\cite[Theorem~3.9.11]{Derksen:Kemper}), until in 2011 Peter
Symonds~\cite{Symonds:2011} proved a bound that had been conjectured
for a while. His paper tells us that if $n \ge 2$ is the number of
variables and $|G| \ge 2$, then $\R^G$ is generated by homogeneous
invariants of degrees $\le n \bigl(|G| - 1\bigr)$. The proof makes
heavy use of the paper~\cite{KS:02}, which itself establishes the
fascinating result that the polynomial ring $\R$, viewed as a direct
sum of infinitely many indecomposable $KG$-modules, only contains a
finite number of isomorphism types of such modules. As the title
of~\cite{Symonds:2011} suggests, it not only provides a bound on the
generators of $\R^G$ but also on the algebraic relations between
them. For an (even) better appreciation of Symonds's bound, and for
understanding why it qualifies as ``reasonable,'' it should be pointed
out that it is really a bound on secondary invariants.

This brings us to the topic of primary and secondary invariants, which
we already mentioned in passing before presenting King's algorithm. As
they are covered in Bernd's book, we will just briefly recall them
here. {\bf Primary invariants}, which owe their existence to Noether
normalization, are homogeneous invariants $f_1 \upto f_n$ (again
with~$n$ the number of variables) such that $\R^G$ is finitely
generated as a module over the subalgebra $A := K[f_1 \upto f_n]$
generated by them. After having chosen primary invariants, we call
homogeneous generators $g_1 \upto g_m$ of $\R^G$ as an $A$-module {\bf
  secondary invariants}. Together, the primary and secondary
invariants form a generating system of $\R^G$ as an algebra. Now what
Symonds's bound really says is that if the secondary invariants are
chosen to {\em minimally} generate $\R^G$ as an $A$-module, then
\begin{equation} \label{eqSymonds}%
  \deg(g_i) \le \sum_{i=1}^n (\deg(f_i) - 1).
\end{equation}
A typical example are the invariants of the alternating group $A_n$
acting naturally by permuting the~$x_i$, which are treated in Bernd's
book~\cite[Proposition~1.1.3]{stu}. In this case the elementary
symmetric polynomials can be taken as primary invariants and, in
characteristic $\ne 2$, secondary invariants are given by $g_1 = 1$
and $g_2 = \prod_{i < j} (x_i - x_j)$. So here Symonds's bound for
secondary invariants is sharp, as it quite often is. It is, however,
seldom sharp as a bound for algebra generators, since typically the
secondary invariant(s) of top degree can be chosen as a product of
lower-degree secondary invariants. But precisely because it is a bound
on secondary invariants, it is extremely useful for computations,
since in the modular case the state-of-the-art algorithms for
computing $\R^G$ do proceed by consecutively computing primary and
secondary invariants.

How does this lead to the bound $n\bigl(|G| - 1\bigr)$ stated above?
To see this, we need to recall how primary invariants are
characterized geometrically: homogeneous invariants $f_1 \upto f_n$
are primary invariants if and only if the affine variety
$\Var_{\overline{K}^n}(f_1 \upto f_n)$ over an algebraic closure of
$K$ defined by them consists of the point $(0 \upto 0)$ alone (see,
for example, \cite[Proposition~5.3.7]{lsm}).
We will often use the less context-specific term {\bf homogeneous
  system of parameters} ({\bf hsop}) for the set $\{f_1 \upto f_n\}$.
Also notice that an equivalent condition to the above criterion is
that $\Var_{\overline{K}^n}(f_1 \upto f_n)$ is zero-dimensional.  This
can be refined: a set $\{f_1 \upto f_k\}$ of homogeneous polynomials,
with $k \le n$, can be extended to an hsop if and only if the variety
it defines has dimension $n - k$, the least dimension possible. In
this case we speak of a {\bf partial homogeneous system of parameters}
({\bf phsop}). \label{phsop} It is these geometric tools that are at
the heart of algorithms for constructing primary invariants, as those
given in Bernd's book~\cite[Section~2.5]{stu} or the
refinement~\cite{kem:g}, which still seems to be the state of the art.

In particular, at least if the field $K$ has enough elements, we can
form primary invariants by taking orbit products of linear forms, if
only these forms are chosen ``in general position.'' This method is
referred to as {\em Dade's algorithm}, and more details can be found
in Bernd's book \cite[Subroutine~2.5.12]{stu}. This brings us back to
Symonds's bound. The primary invariants constructed by Dade's
algorithm have degree $\le |G|$, so the bound $n \bigl(|G| - 1\bigr)$
follows directly from~\cref{eqSymonds}. What about the assumption that
$K$ has enough elements?  This is actually a without-loss assumption,
since it can be shown that the maximal degree of a minimal generating
system of $\R^G$ does not change when the ground field $K$ is
extended. \\

Having reported progress on degree bounds in the modular case, let us
now address a different issue, the Cohen-Macaulay property. In the
case of invariant rings, this property can be defined in a very simple
way, which also makes it immediately clear that it is a property worth
having: $\R^G$ is Cohen-Macaulay if it is free as a module over the
subalgebra generated by some (or, equivalently, every) hsop. Now one
of the prominent results in nonmodular invariant theory is that in the
nonmodular case, $\R^G$ always is Cohen-Macaulay (see Bernd's
book~\cite[Theorem~2.3.5]{stu}). And once again, this often breaks
down in the modular case. At the time of writing Bernd's book, results
on the Cohen-Macaulay property in modular invariant theory were rather
haphazard: there were just a few examples known where the
Cohen-Macaulay property failed, and a few others where it held. The
only notable (and, indeed very remarkable) result was the
paper~\cite{ES} by Ellingsrud and Skjelbred, which answers the
question completely in the case of cyclic $p$-groups. Later, Campbell,
Geramita, Hughes, Shank, and Wehlau \cite{CGHSW:b} considered the case
of {\em vector invariants}, which means that a given linear action of
a group $G$ is replicated on several sets of variables; so if
$\sigma(x_i) = a_{1,i} x_1 + \cdots + a_{n,i} x_n$, then the action on
the new variables $x_{i,j}$ is by
$\sigma(x_{i,j}) = a_{1,i} x_{1,j} + \cdots + a_{n,i} x_{n,j}$,
for~$j$ between~$1$ and some~$k$. An example with $k = 3$ is the
action of $C_2$ at the beginning of the section. The result
in~\cite{CGHSW:b} says that if $G$ is a $p$-group acting nontrivially
and if $k \ge 3$, then the ring of vector invariants is not
Cohen-Macaulay. In fact, the $C_2$-action just recalled is not only
the most accessible example where Noether's bound fails, but also
where the Cohen-Macaulay property fails. The following result came out
later in the same year as~\cite{CGHSW:b}.

\begin{theorem}[Kemper~\cite{kem:h}] \label{tCM}%
  If $\R^G$ is Cohen-Macaulay, then $G$ is generated by $p'$-elements
  (i.e., by elements of order not divisible by~$p$) and by
  bireflections (i.e., by elements that fix a subspace of
  codimension~$2$).
\end{theorem}

The term bireflection is modeled after the concept of a reflection,
which is a linear transformation fixing a subspace of
codimension~1. Notice that ``fixing a subspace of codimension~$2$''
allows that the fixed space actually has smaller codimension; so in
particular every reflection is also a bireflection. On the other hand,
if we consider vector invariants of $k \ge 3$ sets of variables (the
case of~\cite{CGHSW:b}), then the only bireflection is the identity;
so~\cite{CGHSW:b} is contained in \cref{tCM}. It should be mentioned
that, unfortunately, the converse of \cref{tCM} is not true, and the
quest for a group or representation theoretic if-and-only-if criterion
for the Cohen-Macaulay property is still one of the holy grails in
modular invariant theory.

The paper~\cite{kem:h} proves \cref{tCM} within a broader framework,
so it may be worthwhile to present a proof here that is shorter and
more streamlined, and thus better suited to make the arguments
explicit. As we go along in the proof, we will introduce the relative
trace map and recall the concept of a regular sequence. Both are
interesting in themselves.

This is an outline of the proof: We assume that $G$ is {\em not}
generated by $p'$-elements and bireflections, and derive the existence
of a certain normal subgroup $N$ in \cref{lN}. \cref{lTrace} then
reveals the variety defined by the image $I$ of the relative trace
with respect to $N$, and \cref{lPhsop} tells us that $I$ contains a
phsop (see on \cpageref{phsop}) of length~$3$. Finally in
\cref{lEndgame}, where the endgame of the proof plays out, we show
that such a phsop is not a regular sequence, implying that $\R^G$ is
not Cohen-Macaulay. The first lemma is purely group theoretic and
holds, like everything else, under the assumption that $G$ is not
generated by $p'$-elements and bireflections.

\begin{lemma} \label{lN}%
  There is a normal subgroup $N \subset G$ of index~$p$ that contains
  all bireflections.
\end{lemma}

\begin{proof}
  The subgroup $N_0 \subseteq G$ generated by all $p'$-elements and
  bireflections is normal. Since the order of every element from
  $G/N_0$ is a power of~$p$, $G/N_0$ is a $p$-group. By assumption, it
  is nontrivial, so it has a subgroup $N/N_0$ of index~$p$. But a
  subgroup of index~$p$ in a $p$-group is always normal (see
  \cite[Kapitel~I, Satz~8.9]{hup}).
\end{proof}

Since in modular invariant theory we cannot average over the group,
what is left to do is just summing. The resulting map $\R \to \R^G$ is
called the {\bf trace map} or {\bf transfer}. It has a relative
version: for $H \subseteq G$ a subgroup, the {\bf relative trace} (or
{\bf relative transfer}) is the map
\[
  \Tr_{G/H}\mbox{:} \ \R^H \to \R^G, \ f \mapsto \sum_{\sigma \in G/H}
  \sigma(f),
\]
where the summation is over a set of left coset representatives. This
is surjective if $p \nmid [G:H]$, but how big is its image (which is
an ideal in $\R^G$) otherwise? An answer in geometric terms (i.e.,
determining the radical ideal of the image) can be found
in~\cite[Lemma~1.1]{Lorenz.Pathak}. What we need here is the special
case $H = N$ with $N$ from \cref{lN}. For simplicity (and since it is
an assumption that can be made without loss of generality) we will
assume $K$ to be algebraically closed. The following lemma determines
the variety in $V := K^n$ defined by the ideal
\[
  I_{G/N} := \Tr_{G/N}\bigl(\R^N\bigr) \subseteq \R^G.
\]

\begin{lemma} \label{lTrace}%
  $\Var_V(I_{G/N}) = \bigcup_{\sigma \in G \setminus N} V^\sigma$. In
  particular, since all bireflections are contained in $N$, the
  variety has dimension $\le n - 3$.
\end{lemma}

\begin{proof}
  For the inclusion ``$\supseteq$'' assume a vector $v \in V$ is fixed
  by some $\sigma \in G \setminus N$. The class of $\sigma$ generates
  $G/N$, so for $f \in \R^N$ we obtain
  \[
    \bigl(\Tr_{G/N}(f)\bigr)(v) = \sum_{i=0}^{p-1}
    \bigl(\sigma^i(f)\bigr)(v) = \sum_{i=0}^{p-1}
    f\bigl(\sigma^{-i}(v)\bigr) = \sum_{i=0}^{p-1} f(v) = 0,
  \]
  since $K$ has characteristic~$p$. For the reverse inclusion, assume
  $\sigma(v) \ne v$ for all $\sigma \in G \setminus N$, so the sets
  $\{\sigma(v) \mid \sigma \in G \setminus N\}$ and
  $\{\tau(v) \mid \tau \in N\}$ are disjoint. By interpolation, we can
  find a polynomial $f \in \R$ that is constantly zero on the first
  set but constantly~$1$ on the second. Replacing~$f$ by
  $\prod_{\tau \in N} \tau(f)$, we may assume $f \in \R^N$. Then
  $\bigl(\Tr_{G/N}(f)\bigr)(v) = 1$, so $v \notin \Var_V(I_{G/N})$.
\end{proof}

\begin{lemma} \label{lPhsop}%
  $I_{G/N}$ contains a phsop of length~3.
\end{lemma}

\begin{proof}
  We will prove the following, more general statement, using induction
  on~$k$: if $I \subseteq \R^G$ is a homogeneous ideal such that
  $\Var_V(I)$ has dimension $\le n - k$, then $I$ contains a phsop of
  length~$k$. There is nothing to show for $k = 0$, so we can assume
  $k \ge 1$ and, by induction, that we have already found a phsop
  $f_1 \upto f_{k-1} \in I$. This means that the prime ideals
  $Q_1 \upto Q_m$ in $\R$ that minimally lie over
  $(f_1 \upto f_{k-1})$ all have dimension $n - k + 1$ (some readers
  may prefer to say that the height is~$k - 1$). So by hypothesis, the
  ideal $J := \R \cdot I$ generated by $I$ is contained in none of the
  $Q_i$, so by prime avoidance (see~\cite[Lemma~3.3]{eis}), there is a
  homogeneous $g \in J$ with $g \notin Q_i$ for all~$i$. Since the
  $G$-action permutes the $Q_i$, this is also true for all
  $\sigma(g)$, so also the product
  $f_k := \prod_{\sigma \in G} \sigma(g)$ avoids all $Q_i$. This
  implies that $f_1 \upto f_k$ are a phsop.

  So what is left to show is $f_k \in I$. Since $g \in J$ we have
  $g = \sum_{i=1}^r h_i g_i$ with $g_i \in I$ and $h_i \in \R$. With
  $t_1 \upto t_r$ new variables, we have a $G$-equivariant map $\phi$:
  $\R[K][t_1 \upto t_r] \to \R$ of $\R$-algebras sending~$t_i$ to~$g_i$
  (with trivial $G$-action on the~$t_i$). So
  $\phi(h_1 t_1 + \cdots + h_r t_r) = g$ and hence
  \[
    \phi\left(\prod_{\sigma \in G} \sigma(h_1 t_1 + \cdots + h_r
      t_r)\right) = \prod_{\sigma \in G} \sigma(g) = f_k.
  \]
  But the product of the $\sigma(h_1 t_1 + \cdots + h_r t_r)$ lies in
  $\bigl(\R[K][t_1 \upto t_r]\bigr)^G = \linebreak \R^G[t_1 \upto
  t_r]$, so applying~$\phi$ to it yields an element of $I$ since
  $g_i \in I$. This completes the proof.
\end{proof}

\begin{rem} \label{rPhsop}%
  In the above proof we really showed the following statement: every
  phsop in a homogeneous ideal $I \subseteq \R^G$ with
  $\dim\bigl(\Var_V(I)\bigr) \le n - k$ can be extended to a phsop in
  $I$ of length~$k$.
\end{rem}

Before presenting the final step in the proof of \cref{tCM} we need to
recall the concept of a regular sequence. In our context this can be
defined as a sequence $f_1 \upto f_k \in \R^G$ of nonconstant,
homogeneous invariants such that for each~$i$, the map
$\R^G/(f_1 \upto f_{i-1}) \to \R^G/(f_1 \upto f_{i-1})$ given by
multiplication with~$f_i$ is injective. It is easily seen that if
$\R^G$ is Cohen-Macaulay, then every phsop is a regular sequence. (The
converse also holds, but we do not need it here.) In particular, every
phsop is an $\R$-regular sequence, i.e., regular when regarded as a
sequence in $\R$. It is interesting that for proving the following
lemma, the $\R$-regularity is actually used to show that
$\R^G$-regularity fails.

\begin{lemma} \label{lEndgame}%
  A phsop $f_1,f_2,f_3 \in I_{G/N}$ of length~3 is not a regular
  sequence. Therefore $\R^G$ is not Cohen-Macaulay.
\end{lemma}

\begin{proof}
  Choose an element $\sigma \in G \setminus H$. By the definition of
  $I_{G/N}$ we have $f_i = \sum_{j=0}^{p-1} \sigma^j(h_i)$ with
  $h_i \in \R^N$. Since $N$ is a normal subgroup we have
  $\sigma^j(h_i) \in \R^N$. In characteristic~$p$, we have the
  polynomial identity
  $1 + t + \cdots + t^{p-1} = (1 - t) \bigl(1 + 2 t + \cdots + (p-1)
  t^{p-2}\bigr)$. So setting
  $g_i := - \sum_{j=1}^{p-1} j \sigma^{j-1}(h_i)$ yields
  \begin{equation} \label{eqFi}%
    f_i = \sigma(g_i) - g_i \quad \text{with} \quad g_i \in \R^N
    \qquad (i = 1,2,3).
  \end{equation}
  A consequence is that for $1 \le i < j \le 3$ the
  $u_{i,j} := f_i g_j - f_j g_i \in \R^N$ satisfy
  \[
    \sigma(u_{i,j}) = f_i \sigma(g_j) - f_j \sigma(g_i)
    \underset{\cref{eqFi}}{=} f_i (f_i + g_i) - f_j (f_i + g_i) =
    u_{i,j},
  \]
  so in fact $u_{i,j} \in \R^G$. The
  claim that the sequence $f_1,f_2,f_3$ is not $\R^G$-regular will
  follow from the relation
  \[
    f_1 u_{2,3} - f_2 u_{1,3} + f_3 u_{1,2} = \det
    \begin{pmatrix}
      f_1 & f_2 & f_3 \\
      f_1 & f_2 & f_3 \\
      g_1 & g_2 & g_3
    \end{pmatrix} = 0.
  \]
  Indeed, assuming regularity would yield $u_{1,2} = (f_1,f_2)$, i.e.,
  $u_{1,2} = f_1 \widetilde g_2 - f_2 \widetilde g_1$ with
  $\widetilde g _i \in \R^G$. With the definition of~$u_{1,2}$ this
  gives $f_1 (g_2 - \widetilde g_2) = f_2 (g_1 - \widetilde
  g_1)$. Since $f_1,f_2$ is a $\R$-regular sequence, we obtain
  $g_1 - \widetilde g_1 \in (f_1)$, so $g_1 - \widetilde g_1 = f_1 d$
  with $d \in \R$. By~\cref{eqFi} and since $\widetilde g_1 \in \R^G$,
  this implies
  \[
    f_1 \bigl(\sigma(d) - d\bigr) = \sigma(g_1 - \widetilde g_1) -
    (g_1 - \widetilde g_1) = f_1.
  \]
  Now because $f_1$ is $\R$-regular, we obtain $\sigma(d) - d = 1$,
  which is a contradiction since for every polynomial the homogeneous
  part of degree~$0$ is an invariant.
\end{proof}

With this, the proof of \cref{tCM} is complete. The proof can be
framed in homological terms: there is a nonzero element in the first
cohomology $H^1(G/N,R^N)$, and this cohomology group is annihilated
when multiplied by elements from the relative trace ideal
$I_{G/N}$. Now from the fact that a phsop of length~3 annihilates a
nonzero class in $H^1(G,R)$, it can be concluded that $\R^G$ fails to
be Cohen-Macaulay. This approach, with using higher cohomology as
well, has brought forth some further results, such as:
\begin{itemize}
\item If $G$ acts by the regular representation, then $\R^G$ is
  Cohen-Macaulay if and only if $G \cong C_2$, $G \cong C_3$, or
  $G \cong C_2 \times C_2$ (see \cite{kem:h}).
\item If $G$ acts simultaneously on~$k$ sets of variables (the case of
  {\em vector invariants} that was also treated in \cite{CGHSW:b}),
  then $\R^G$ is not Cohen-Macaulay for~$k$ large enough (see
  \cite{kem:h}). In fact, with rising~$k$, the {\em Cohen-Macaulay
    defect}
  $\dim\bigl(\R^G\bigr) - \operatorname{depth}\bigl(\R^G\bigr)$ tends
  to infinity (see \cite{Gordeev.Kemper}).
\item If $G$ acts as a permutation group and $|G|$ is not divisible
  by~$p^2$ (the {\em mildly modular} case), the exact depth of $\R^G$
  was determined in~\cite{kem:flat}. For example, if
  $G \not \cong C_2$ acts by the regular representation,
  $\operatorname{depth}\bigl(\R^G\bigr) = \frac{|G|}{p} + 2$,
  consistent with the result mentioned above. As another example, the
  Cohen-Macaulay defect of vector invariants of $S_n$, with
  $p \le n < 2 p$, acting on $k \ge 2$ sets of variables, is
  $(k-2) (p-1)$.
\end{itemize}
In the above listing we have always assumed that the
characteristic~$p$ divides $|G|$ and $G$ acts faithfully.

Much more recently, Ben Blum-Smith and Sophie Marques~\cite{BSM:2018}
scrutinized the case of permutation actions further and achieved the
following result: for a permutation group $G$ (meaning that $G$
permutes the variables~$x_i$) the invariant rings in all
characteristics are Cohen-Macaulay if and only if $G$ is generated by
bireflections. (Notice that in the case of permutation actions,
bireflections are transpositions, double-transpositions and 3-cycles.)
This is remarkable not only because of its simplicity, but also since
it reaches the gold standard of directly translating a group or
representation theoretic property of the action to an algebraic
property of the invariant ring.

So far, we have looked at two nice features of nonmodular invariant
theory that are missing in the modular case: Noether's degree bound
and the Cohen-Macaulay property. A further feature, which is
especially useful for computations, is Molien's formula, presented in
Bernd's book~\cite[Theorem~2.2.1]{stu}. It affords the computation of
the Hilbert series of $\R^G$ without computing a single invariant. But
again, in the modular case Molien's formula is absent. What was
present at the time of Bernd's book was a paper by Almkvist and
Fossum~\cite{AF} that provides formulas for the Hilbert series in the
case of an indecomposable action of the cyclic group $C_p$ of
order~$p$. In a paper with Ian Hughes~\cite{Hughes:Kemper:b}, we
picked up this thread and were surprised to find that the methods can
be carried further to the {\em mildly modular} case, i.e., the case
where $|G|$ is divisible by~$p$ but not by~$p^2$. The
paper~\cite{Hughes:Kemper:b} gives a recipe for computing the Hilbert
series of $\R^G$ which, like Molien's formula, does the job without
considering any invariant. Also in terms of computational cost, the
recipe is very similar to Molien's formula, it just takes longer to
explain (or implement) it. The
book~\cite[Section~3.4.2]{Derksen:Kemper:2015} presents an outline, so
let us direct interested readers there for details. We hit multiple
roadblocks, however, when trying to move anywhere beyond the mildly
modular case. Perhaps a better understanding of Karagueuzian's and
Symonds's methods in~\cite{KS:02} will open new pathways to the a
priori computation of Hilbert series in modular invariant theory.


\section{Separating invariants of finite groups} \label{sSepFin}

One of the main purposes of invariants, perhaps even their raison d'\^
etre, is to separate groups orbits. The theme permeates Bernd's book,
and it prompted the development of Geometric Invariant Theory.
So it seems natural to consider the separation properties of given
invariants, which leads to the concept of separating invariants. Let
us recall the definition. A subset $S \subseteq \R^G$ is called {\bf
  separating} if for any two points $v,w \in V := K^n$ we have: if
there is an invariant $f \in \R^G$ such that $f(v) \ne f(w)$, then
there is a $g \in S$ with $g(v) \ne g(w)$.

Here $G$ need not be finite, although we are considering the finite
case in this section. Clearly every set of generating invariants is
also separating. Perhaps the easiest example that reveals that a
separating set can really be smaller than a generating one is the
action of a cyclic group $C_n$ on $\CC^2$ by the scalar matrices
$e^{2 k \pi i/n} \cdot \left(\begin{smallmatrix} 1 & 0 \\ 0 &
    1 \end{smallmatrix}\right)$. A minimal {\em generating} set
consists of all~$n + 1$ monomials of degree~$n$; however
$f_1 := x_1^n$, $f_2 := x_1^{n-1} x_2$, and $f_3 := x_2^n$ form a
separating set. This follows from the formula
\[
  x_1^{n-k} x_2^k = \frac{f_2^k}{f_1^{k-1}}.
\]
As we are dealing with a weakening of the concept of generating the
invariant ring, this offers some hope that the bad behavior of
generating sets in the modular case might be mitigated by considering
separating sets instead. And, indeed, Noether's degree bound is
reinstated by the following elementary result.

\begin{theorem}[Noether's degree bound for separating
  invariants] \label{tNoetherSep}%
  Let $G$ be a finite group acting on $\R$ by linear transformations
  of the~$x_i$. With additional variables~$t$ and~$y$, form the
  polynomial
  \[
    F(t,y) := \prod_{\sigma \in G} \left(y - \sum_{i=1}^n \sigma(x_i)
      t^{i-1}\right) \in \R^G[t,y].
  \]
  If all the coefficients of $F$ agree on two points $v,w \in K^n$,
  then~$v$ and~$w$ are in the same $G$-orbit. In particular, the
  coefficients of $F$ form a separating set consisting of homogeneous
  invariants of degree $\le |G|$, and all $G$-orbits can be separated
  by invariants.
\end{theorem}

\begin{proof}
  The hypothesis gives
  \[
    \prod_{\sigma \in G} \left(y - \sum_{i=1}^n
      \bigl(\sigma(x_i)\bigr)(v) t^{i-1}\right) = \prod_{\sigma \in G}
    \left(y - \sum_{i=1}^n \bigl(\sigma(x_i)\bigr)(w) t^{i-1}\right).
  \]
  So there is a $\sigma \in G$ such that
  $\sum_{i=1}^n x_i(w) t^{i-1} = \sum_{i=1}^n
  \bigl(\sigma(x_i)\bigr)(v) t^{i-1}$. Of course~$x_i$ is nothing but
  the $i$-th coordinate functional on $K^n$. For all~$i$ we obtain
  \[
    x_i(w) = \bigl(\sigma(x_i)\bigr)(v) =
    x_i\bigl(\sigma^{-1}(v)\bigr),
  \]
  which implies $w = \sigma^{-1}(v)$. So the first statement is
  proved, and the other statements are now clear.
\end{proof}

Of course the theorem contains a simple method for constructing
separating invariants that does not require any Gr\"obner basis
computation at all. However, even for moderate-sized groups the
computational cost of expanding the product and extracting
coefficients is enormous, as is the number of invariants that the
method produces. \\

By now, we have reported progress on almost all topics treated in the
chapter on finite group actions in Bernd's book, with one notable
exception: reflection groups. These are groups generated by
reflections, where a reflection is an element of $G$ that fixes a
codimension-$1$ subspace of $K^n$, including reflections of orders
other than~$2$ (if the ground field $K$ allows them) and even elements
acting trivially. Also very well-known is the result that in the
nonmodular case, the invariant ring $\R^G$ can be generated by~$n$
invariants (equivalently, is isomorphic to a polynomial ring) if and
only if $G$ is a reflection group (see Bernd's
book~\cite[Theorem~2.4.1]{stu}). As readers will be expecting by now,
this breaks down in the modular case. But not completely: while there
are many examples of modular reflection groups whose invariant ring is
{\em not} a polynomial ring, it is still true that for $\R^G$ to be
polynomial, $G$ has to be a reflection group. This implication is
often attributed to Serre, but in~\cite{Serre68} he says that it was
proved in 1955 by Chevalley. Meanwhile, Chevalley's article
\cite{Chevalley} does appear to not contain the result, and
Benson~\cite{bens} gives the citation~\cite[Chapitre~5, \S~5,
Exercice~7]{bou:81}, where there is indeed a blueprint of a proof,
using purity of the branch locus.

So far, this has nothing to do with separating invariants. But since
generating invariants are always separating, a stronger result would
be that if there exist~$n$ separating invariants, then $G$ has to be a
reflection group. And this is precisely what Emilie Dufresne proved in
her amazing paper~\cite{Dufresne:2009}. It is striking that Dufresne's
result has a rather short but elegant proof. But before turning to
that, let us present the precise statement.

\begin{theorem}[Dufresne~\cite{Dufresne:2009}] \label{tDufresne}%
  Assume that $K$ is algebraically closed and that there is a
  separating set of~$n$ invariants, with~$n$ the number of
  variables. Then $G$ is a reflection group.
\end{theorem}

Dufresne's paper~\cite{Dufresne:2009} also offers an example of an
invariant ring that has~$n$ separating invariants, but does {\em not}
have~$n$ generating invariants. What about the hypothesis that $K$ be
algebraically closed? The following example shows that the theorem
fails without that hypothesis. It also shows that when it comes to
separating invariants, the ground field $K$ makes a difference.

\begin{ex} \label{exDufresne}%
  The group $G := \langle A\rangle \subset \GL_2(\RR)$ generated by
  $A := \left(\begin{smallmatrix} 0 & 1 \\ -1 &
      -1 \end{smallmatrix}\right)$ (the companion matrix of the third
  cyclotomic polynomial) is not a reflection group. The transformation
  given by $A$ maps~$x_1$ to~$- x_2$ and~$x_2$ to~$x_1 - x_2$, and
  thus, with $\omega := e^{2 \pi i/3}$ and
  $y_j := x_1 + \omega^j x_2 \in \R[\CC]$ ($j = 1,2$), it maps~$y_j$
  to~$\omega^j y_j$. Therefore
  \[
    f_1 := y_1^3 + y_2^3 \quad \text{and} \quad f_2 := \sqrt{-3}
    \bigl(y_1^3 - y_2^3\big)
  \]
  are invariant under $G$, and also under complex conjugation. So
  $f_1,f_2 \in \R[\RR]^G$. We claim that the~$f_i$ are separating
  invariants. Observe that they are also invariant under the
  transformation given by
  $D:= \left(\begin{smallmatrix} \omega & 0 \\ 0 &
      \omega \end{smallmatrix}\right)$, so
  $f_i \in \R[\CC]^{\widetilde G}$ with
  $\widetilde G := \langle A,D\rangle$. In fact, $\widetilde G$ is
  generated by the transformations
  $\left(\begin{smallmatrix} y_1 \\ y_2\end{smallmatrix}\right)
  \mapsto \left(\begin{smallmatrix} \omega y_1 \\
      y_2\end{smallmatrix}\right)$ and
  $\left(\begin{smallmatrix} y_1 \\ y_2\end{smallmatrix}\right)
  \mapsto \left(\begin{smallmatrix} y_1 \\ \omega
      y_2\end{smallmatrix}\right)$. So it is a reflection group, and
  \[
    \R[\CC]^{\widetilde G} = \CC[y_1,y_2]^{\widetilde G} =
    \CC[y_1^3,y_2^3] = \CC[f_1,f_2].
  \]
  Being generating invariants, the~$f_i$ are also separating
  invariants in $\R[\CC]^{\widetilde G}$ and by \cref{tNoetherSep},
  they separate $\widetilde G$-orbits in $\CC^2$. But our claim was
  that they are separating invariants in $\R[\RR]^G$. To prove this,
  let~$v,w \in \RR^2$ such that $f_i(v) = f_i(w)$ for $i = 1,2$. Then
  by what we have just seen there exists
  $\widetilde \sigma \in \widetilde G$ such that
  $w = \widetilde \sigma(v)$. We can write
  $\widetilde \sigma = \omega^i \cdot \sigma$ with $\sigma \in G$
  and~$i \in \ZZ$, so $w = \omega^i \sigma(v)$. But~$w$
  and~$\sigma(v)$ have components in $\RR$, so $\omega^i = 1$
  or~$v = w = 0$. In both cases $G(v) = G(w)$, and our claim is
  proved. Of course the~$f_i$ are {\em not} separating invariants in
  $\R[\CC]^G$.

  To sum up, this example provides a finite group
  $G \subset \GL_2(\RR)$ which is not a reflection group, but there
  exists a separating set of two invariants.
\end{ex}

Over finite ground fields, there can even exist separating sets with
fewer than~$n$ elements (see
\cite[Theorem~1.1]{Kemper:Lopatin:Reimers:2022}).

Even though \cref{tDufresne} requires $K$ to be algebraically closed,
the implication ``$\R^G$ polynomial $\Rightarrow$ $G$ reflection
group'' follows from it in full generality, i.e., without the
algebraic closedness hypothesis. In fact, if there are are~$n$
invariants generating $\R^G$, they also generate $\R[\overline{K}]^G$
and are therefore separating for the $G$-action on $\overline{K}^n$.
From this, \cref{tDufresne} yields that $G$ is a reflection
group.

In a nutshell, Dufresne's proof of \cref{tDufresne} consists of the
crucial observation that the graph of the action is connected in
codimension~$1$ if and only if $G$ is a reflection group (see
\cref{pDufresne} below), combined with Hartshorne's connectedness
theorem. Let us explain. By the {\bf graph of the action} we mean the
set
$\Gamma := \bigl\{\bigl(v,\sigma(v)\bigr) \mid v \in V, \ \sigma \in
G\bigr\} \subset V \times V$, where $V := K^n$. Since $G$ is finite,
this is (Zariski-)closed in $V \times V$, and, by \cref{tNoetherSep},
a pair $(v,w)$ of points lies in $\Gamma$ if and only if all
invariants take the same value on~$v$ and~$w$. Furthermore, a variety
is called {\bf connected in codimension~$k$} if it is connected even
after removing a closed subset of codimension~$> k$. (A bit of care
must be taken to get the definition of codimension right if the
variety has irreducible components of different dimensions, but this
is not the case for the graph of the action.) For example, two planes
meeting in a line are connected in codimension~$1$, but two planes
meeting in a point are not. {\em Not} to be connected in
codimension~$k$ means that the variety can be written as a union of
two closed subsets, none contained in the other, whose intersection
has codimension~$> k$. For want of a better term, let us call such a
pair of closed subsets a {\em witness of disconnectedness}.

\begin{prop}[Dufresne~\cite{Dufresne:2009}] \label{pDufresne}%
  The graph of the action is connected in codimension~$1$ if and only
  if $G$ is a reflection group. If this is not the case, there is a
  witness $(X,Y)$ of disconnectedness with $X$ and $Y$ given by
  homogeneous ideals.
\end{prop}

\begin{proof}
  The irreducible components of $\Gamma$, the graph of the action, are
  the sets
  \[
    Z_\sigma := \bigl\{\bigl(v,\sigma(v)\bigr) \mid v \in V\bigr\} \cong V
    \quad \text{for} \ \sigma \in G,
  \]
  and their intersections are
  \begin{equation} \label{eqZ}%
    Z_\sigma \cap Z_\tau = \bigl\{\bigl(v,\sigma(v)\bigr) \mid v \in V, \
    \sigma(v) = \tau(v)\bigr\} \cong V^{\sigma^{-1} \tau}.
  \end{equation}
  If $\Gamma$ is not connected in codimension~$1$, then there is a
  witness $(X,Y)$ of disconnectedness. Each irreducible component
  $Z_\sigma$ is contained in $X$ or in $Y$, but not in both. So
  writing $\sigma \sim \tau$ if $Z_\sigma$ and $Z_\tau$ are both
  contained in $X$ or both contained in $Y$ defines an equivalence
  relation on $G$. If $\sigma \not\sim \tau$, then
  $\codim(Z_\sigma \cap Z_\tau) > 1$, which by~\cref{eqZ} means that
  $\sigma^{-1} \tau$ is {\em not} a reflection. Using transitivity, we
  conclude that if~$\sigma$ and~$\tau$ lie in the same left coset of
  the subgroup $H \subseteq G$ generated by all reflections, then
  $\sigma \sim \tau$. But not all group elements can be equivalent,
  since otherwise $\Gamma = X$ or $\Gamma = Y$. it follows that
  $H \subsetneqq G$, so $G$ is not a reflection group.

  This also works backwards: If $H \subsetneqq G$, form
  $X := \bigcup_{\sigma \in H} Z_\sigma$ and
  $Y := \bigcup_{\tau \in G \setminus H} Z_\tau$ (or lump irreducible
  components corresponding to cosets of $H$ together in any other
  way). Then $(X,Y)$ is a witness for disconnectedness. Since every
  $Z_\sigma$ is given by homogeneous linear equations, $X$ and $Y$ are
  given by homogeneous ideals.
\end{proof}

\begin{rem} \label{rDufresne}%
  Of course by the same argument we see that $\Gamma$ is connected in
  codimension~$2$ if and only if $G$ is generated by bireflections,
  and so on.
\end{rem}

Hartshorne's connectedness theorem is also about connectedness in
codimension~$1$, but how can we use it? Assume there are~$n$
separating invariants $f_1 \upto f_n$. Taking additional variables
$y_1 \upto y_n$ and setting
$\Delta f_i := f_i(\mathbf x) - f_i(\mathbf y)$, which can be
evaluated at points from $V \times V$, we see that the variety in
$V \times V$ given by the $\Delta f_i$ is precisely $\Gamma$. Since
$\dim(\Gamma) = n$, this implies that $\Gamma$ is a complete
intersection. Therefore
$R := K[\mathbf x,\mathbf y]/(\Delta f_1 \upto \Delta f_n)$ is
Cohen-Macaulay (see \cite[Proposition~18.13]{eis}), and now the
connectedness theorem~\cite[Corollary~2.4 and
Remark~1.3.2]{Hartshorne:1962} tells us that $\Gamma$ (more precisely,
$\Spec(R)$, a potentially nonreduced version of $\Gamma$) is connected
in codimension~$1$. So applying \cref{pDufresne} gives us
\cref{tDufresne}.

If Serre or Hartshorne had been aware of \cref{pDufresne}, it seems
likely that they would have made the connection to the connectedness
theorem (no pun intended) and come up with this simple idea to prove
the implication ``$\R^G$ polynomial $\Rightarrow$ $G$ reflection
group.'' Quitting this idle speculation, let us ask how hard or easy
Defresne's proof really is by presenting a textbook-style variant
proof, which instead of citing Hartshorne's connectedness theorem only
uses knowledge from Bernd's book. We will restrict our attention to
the case that the~$n$ separating invariants are homogeneous. This
restriction is a mild for two reasons: (1) most invariant theorists
(almost) never use inhomogeneous invariants in the first place, and
(2) if there are~$n$ possibly inhomogeneous {\em generating}
invariants, then there are also~$n$ homogeneous generating invariants
(it is left as an exercise to prove this);
thus in order to show the implication ``$\R^G$ polynomial
$\Rightarrow$ $G$ reflection group,'' one may assume homogeneity
without loss of generality.

\begin{proof}[Proof of \cref{tDufresne}]
  As explained above, this proof only treats the case of {\em
    homogeneous} separating invariants, and it does not quote any
  ``heavier machinery'' such as Hartshorne \cite[Corollary~2.4 and
  Remark~1.3.2]{Hartshorne:1962}.


  By hypothesis there are homogeneous invariants $f_1 \upto f_n$ such
  that $\Gamma = \Var_{V \times V}(D)$ with
  $D := (\Delta f_1 \upto \Delta f_n)$. Since $\dim(\Gamma) = n$, the
  $\Delta f_i$ form a phsop in $K[\mathbf x,\mathbf y]$. With
  \cref{pDufresne} in mind, we assume that $\Gamma = X \cup Y$ with
  $X$ and $Y$ closed, given by homogeneous ideals, such that
  $\codim_\Gamma(X \cap Y) \ge 2$. So if we can show that
  $X \subseteq Y$ or $Y \subseteq X$, we are done.

  We have $X = \Var_{V \times V}(g_1 \upto g_l)$ and
  $Y = \Var_{V \times V}(h_1 \upto h_m)$ with~$g_i$ and~$h_j$
  homogeneous. Each product $g_i h_j$ lies in the vanishing ideal of
  $X \cup Y = \Gamma$, which is $\sqrt D$, so there is~$k$ such that
  $g_i^k h_j^k \in D$. The ideals $I := (g_1^k \upto g_l^k) + D$ and
  $J := (h_1^k \upto h_m^k) + D$ are homogeneous, satisfy
  $I \cdot J \subseteq D \subseteq I \cap J$, and define the algebraic
  sets $X$ and $Y$. So it suffices to show $I \subseteq J$ or
  $J \subseteq I$.

  The ideal $I + J$ defines $X \cap Y$, which has dimension
  $\le n - 2 = 2n - (n + 2)$. So by \cref{rPhsop}, the phsop formed by
  the $\Delta f_i$ can be extended by two further polynomials, both
  lying in $I + J$. Since $K[\mathbf x,\mathbf y]$ is Cohen-Macaulay,
  our phsop is a regular sequence of length $n + 2$. Forming the ring
  $R := K[\mathbf x,\mathbf y]/D$ and considering the ideals $I/D$ and
  $J/D$, we have the situation of \cref{lConnectedness} below. This
  tells us that $I/D \subseteq J/D$ or $J/D \subseteq I/D$, and the
  desired inclusion follows.
\end{proof}

The following lemma, which may be called the graded version of
Hartshorne's connectedness theorem in ideal-theoretic form, was used
in the above proof.

\begin{lemma} \label{lConnectedness}%
  Let $R$ be a graded ring with $K := R_0$ a field, and let
  $I,J \subseteq R$ be homogeneous ideals such that
  $I \cdot J = \{0\}$. If $I + J$ contains a regular sequence of
  length~$2$, then $I \subseteq J$ or $J \subseteq I$.
\end{lemma}

\begin{proof}
  We have a regular sequence $a_1,a_2 \in I + J$, so
  \[
    a_1 \cdot (I \cap J) \subseteq (I + J) \cdot (I \cap J) \subseteq I
    \cdot J = \{0\}.
  \]
  Therefore $I \cap J = \{0\}$ since~$a_1$ is not a zero divisor. So the
  map~$\phi$: $R \to R/I \oplus R/J$, $x \mapsto (x + I,x + J)$ is
  injective. With~$\psi$: $R/I \oplus R/J \to R/(I + J)$,
  $(x + I,y + J) \mapsto (x - y) + (I + J)$, the sequence
  \[
    \{0\} \longrightarrow R \stackrel{\phi}{\longrightarrow} R/I
    \oplus R/J \stackrel{\psi}{\longrightarrow} R/(I + J)
    \longrightarrow \{0\}
  \]
  is easily checked to be exact. So \cref{lExact} (see below) yields
  $x,y \in R$ such that $(x - y) + (I + J) = 1 + (I + J)$ and
  $(I + J) \cdot x \subseteq I$ and $(I + J) \cdot y \in J$. The
  equality implies that we are in at least one of the following cases:
  (1) $I$ or $J$ contains an element with nonzero component in
  degree~$0$, (2) $x$ has nonzero component in degree~$0$, or (3) $y$
  has nonzero component in degree~$0$. In case~(1), $I$ or $J$ is
  equal to $R$, and the assertion follows. In case~(2),
  $(I + J) \cdot x \subseteq I$ implies that every homogeneous element
  of $J$ lies in $I$, so $J \subseteq I$, and case~(3) works
  analogously.
\end{proof}

Some readers may recognize the formula $\Ext^1_R(R/I,R) = 0$ in the
assertion of the next lemma, which was used in the proof of
\cref{lConnectedness}.

\begin{lemma} \label{lExact}%
  Let $I \subset R$ be an ideal in a ring such that $I$ contains a
  regular sequence of length~2. Then for every exact sequence
  \[
  \{0\} \longrightarrow R \stackrel{\phi}{\longrightarrow} M
  \stackrel{\psi}{\longrightarrow} R/I \longrightarrow \{0\}
  \]
  of $R$-modules, there exists $m \in M$ with $\psi(m) = 1 + I$ and
  $I \cdot m = \{0\}$. (Equivalently, the sequence splits.)
\end{lemma}

\begin{proof}
  We have a regular sequence $a_1,a_2 \in I$. Choose $m' \in M$ with
  $\psi(m') = 1 + I$. For $x \in I$ we have
  $\psi(x m') = x \psi(m') = 0$. So
  \begin{equation} \label{3eqIm}%
    I \cdot m' \subseteq \im(\phi).
  \end{equation}
  In particular, $a_i m' = \phi(b_i)$ with $b_i \in R$, so
  \[
  \phi(a_2 b_1) = a_2 \phi(b_1) = a_2 a_1 m' = a_1 \phi(b_2) =
  \phi(a_1 b_2).
  \]
  Thus $a_2 b_1 = a_1 b_2$, and the regularity of $a_1,a_2$ implies
  $b_1 = a_1 y$ with $y \in R$. Set $m := m' - \phi(y)$. Then
  $\psi(m) = \psi(m') = 1 + I$. Now let $x \in I$. By~\cref{3eqIm}
  there is $z \in R$ with $x m' = \phi(z)$, so
  $\phi(a_1 z) = a_1 x m' = x \phi(b_1) = \phi(x b_1)$, and we obtain
  $a_1 z = x b_1 = a_1 x y$. By the regularity of~$a_1$, this implies
  $z = x y$, so
  \[
  x m = x m' - x \phi(y) = x m' - \phi(x y) = \phi(z) - \phi(z) = 0.
  \]
  This concludes the proof.
\end{proof}

\cref{tDufresne} was extended a few years later, which led to the
following beautiful result:

\begin{theorem}[Emilie Dufresne and Jack
  Jeffries~\cite{Dufresne:Jeffries:2013}] \label{tDufresneJeffries}%
  Assume that $K$ is algebraically closed and there is a separating
  set of~$n + k - 1$ invariants. Then $G$ is generated by
  $k$-reflections, i.e., by elements fixing a subspace of
  codimension~$k$.
\end{theorem}

Even the corollary that if there are $n + k - 1$ {\em generating}
invariants, then $G$ is a $k$-reflection group, is new. The proof is
bafflingly simple: just replace Hartshorne's connectedness theorem by
Grothendieck's connectedness theorem and combine this with
\cref{rDufresne}. (But some care should be taken since Grothendieck's
result is actually about complete local rings.%
) \\

The topic of separating invariants has attracted some sustained
interest. At the time of writing this, searching articles with
``separating invariants'' in the title on MathSciNet produces 25 hits,
the latest one from 2022. Since comprehensiveness is not among the
main goals of this article, let us desist from giving accounts of them
and move on to the topic of infinite group invariants.

\section{Invariants of linearly reductive groups} \label{sReductive}

I have to report fewer new developments in algorithmic invariant
theory of infinite groups than of finite groups. This is why the
remaining sections of the paper are shorter than the previous
ones. But arguably the developments relating to infinite groups are
more important, since in Classical Invariant Theory and in Geometric
Invariant Theory the focus has always been on infinite groups.

It should be safe to say that the most dramatic development in this
area has been Harm Derksen's discovery~\cite{Derksen:99} of an
algorithm for computing generating invariants for linearly reductive
groups. To shift the narrative from the mathematical to the personal:
I was told that Bernd Sturmfels became completely ecstatic when he
learned about the algorithm for the first time. And for good reasons,
because Derken's algorithm ties in perfectly with the philosophy of
Bernd's book. To quote from Frank Grosshans' MathSciNet review of
Bernd's book: ``The efficacy of the Gr\"obner basis algorithms
presented in this book and their place among the tools of invariant
theory is yet to be decided.'' In this context, Derksen's algorithm
establishes the place of Gr\"obner basis methods in invariant theory
as front and center, once again showing Bernd's remarkable knack for
foreseeing where and how things will happen.

Going back to mathematics, let us explain Derksen's algorithm. We will
skip the proofs, directing readers to the presentation
in~\cite[Section~4.1]{Derksen:Kemper:2015}. The algorithm connects
nicely to the last section since the graph of the action
$\Gamma := \bigl\{\bigl(v,\sigma(v)\bigr) \mid v \in V, \ \sigma \in
G\bigr\} \subset V \times V$ (with $V := K^n$ as before) plays an
important role again. Its vanishing ideal
$D := \Id(\Gamma) \subseteq K[\mathbf x,\mathbf y]$ has come to be
known as the \df{Derksen ideal}, and we will discuss how it can be
computed in a moment. The following result paves the way for Derksen's
algorithm. Before we state it, recall that a linear algebraic group
$G$ is called \df{linearly reductive} if there is a Reynolds operator
$\mathcal R$: $\R \to \R^G$ (i.e., a projection that is constant on
orbits) for every morphic $G$-action by linear transformations of
the~$x_i$. Throughout, we assume $K$ to be algebraically closed.

\begin{theorem}[Derksen~\cite{Derksen:99}] \label{tDerksen}%
  Assume $G$ is linearly reductive and let $f_1 \upto f_m$ be
  homogeneous generators of the Derksen ideal $D$. Then the
  $f_i(\mathbf x,\mathbf 0) \in \R$, obtained by setting $y_j = 0$ for
  all~$j$, generate the Hilbert ideal $\I := \R \cdot \R^G_+$ (see at
  the beginning of \cref{sNonmodular}).
\end{theorem}

So from homogeneous generators of $D$ we obtain homogeneous generators
of $\I$, and by a slight extension of \cref{pHilbertIdeal}, applying
the Reynolds operator to them yields generators of the invariant
ring. Since the Reynolds operator is not always known or easily
implementable, a less elegant but more effective variant consists of
throwing away the generators~$f_i$ and only remembering their
degrees~$d_i$, and then computing bases of the spaces $\R^G_{d_i}$ of
homogeneous invariants of degrees~$d_i$ from scratch, which is a
rather simple matter (see
\cite[Algorithm~4.5.1]{Derksen:Kemper:2015}).

With this, everything comes down to the question of how the Derksen
ideal can be computed. To answer it, we must first specify how the
action of $G$ is given. For this, we make ``a morphic action of a
linear algebraic group by linear transformations of the~$x_i$''
explicit. So $G$ is given as an affine variety in $K^r$ by its
vanishing ideal $I_G = (g_1 \upto g_l) \subseteq K[z_1 \upto z_r]$,
with the $z_i$ new variables. Moreover, the action is given by
$\sigma(x_i) = f_i(\sigma)$ for
$\sigma \in G = \Var(I_G) \subseteq K^r$, where
$f_i = \sum_{j=1}^n a_{i,j} x_j$ with $a_{i,j} \in K[\mathbf z]$, and
$f_i(\sigma) := \sum_{j=1}^n a_{i,j}(\sigma) x_j$. Now it is easy to
see that the ideal
\begin{equation} \label{eqD0}%
  \widehat D := \bigl(g_1 \upto g_l,f_1 - y_1 \upto f_n - y_n\bigr)
  \subseteq K[\mathbf x,\mathbf y,\mathbf z]
\end{equation}
is the vanishing ideal of
$\widehat \Gamma := \bigl\{(v,w,\sigma) \in V \times V \times K^r \mid \sigma
\in G, \ w = \sigma(v)\bigr\}$. This implies that the elimination
ideal $K[\mathbf x,\mathbf y] \cap \widehat D$ is the vanishing ideal
$D = \Id(\Gamma)$ of $\Gamma$. So the Derksen ideal can be computed as
an elimination ideal of $\widehat D$, where $\widehat D$ can be formed
immediately from the data defining the group and the action. It is the
computation of this elimination ideal where Gr\"obner basis methods
come in, and where the bulk of the work of Derksen's algorithm lies.

Now we are ready to present the algorithm.

\begin{alg}[Derksen's algorithm~\cite{Derksen:99}] \label{aDerksen} \mbox{}
  \begin{description}
  \item[\bf Input] A linearly reductive group $G$, given by its
    vanishing ideal
    $I_G = (g_1 \upto g_l) \linebreak \subseteq K[z_1 \upto z_r]$, and
    a morphic $G$-action given by polynomials
    $f_1 \upto f_n \in K[x_1 \upto x_n]$ as described above.
  \item[\bf Output] A generating set of the invariant ring $\R^G$ as
    an algebra.
  \end{description}
  \begin{enumerate}[label=(\arabic*)]
  \item \label{aDerksen1} Form the ideal
    $\widehat D \subseteq K[\mathbf x,\mathbf y,\mathbf z]$ as in
    Equation~\cref{eqD0}.
  \item \label{aDerksen2} Using Gr\"obner basis methods, compute the
    elimination ideal \linebreak
    $D := K[\mathbf x,\mathbf y] \cap \widehat D$. Let $f_1 \upto f_m$
    be homogeneous generators of $D$.
  \item \label{aDerksen3} For $i = 1 \upto m$, set $d_i := \deg(f_i)$
    and calculate a basis $B_i$ of the space $\R^G_{d_i}$ of
    homogeneous invariants of degree~$d_i$. For this, Algorithm~4.5.1
    from~\cite{Derksen:Kemper:2015} may be used. Alternatively, if the
    Reynolds operator has been implemented, set
    $B_i :=\bigl\{\mathcal R\bigl(f_i(\mathbf x,\mathbf
    0)\bigr)\bigr\}$.
  \item \label{aDerksen4} Now $B_1 \cup \cdots \cup B_m$ is a (usually
    not minimal) generating set of $\R^G$.
  \end{enumerate}
\end{alg}

Recall that in characteristic~$0$, all reductive groups are linearly
reductive, which includes the classical groups. So Derksen's algorithm
has brought much of Classical Invariant Theory into the realm of
algorithmic computability. But practically, its reach is not
unlimited, since it requires a huge Gr\"obner basis computation. For
example, the more advanced known results about invariants of binary
forms (see the account in \cite[Example~2.1.2]{Derksen:Kemper:2015})
are out of reach for Derksen's algorithm. Still, Derksen's algorithm
is a great tool if you need to know a particular invariant ring that
cannot be found in the literature, or if the algorithm happens to
perform faster than your search of the literature.

\section{Other infinite groups} \label{sInfinite}

In this section we take a brief look at infinite groups that are not
linearly reductive. These fall in two categories: reductive groups in
positive characteristic, and nonreductive groups. Typical examples for
the first case are the classical groups, and for the second case the
additive group. The section could also be titled ``other ways to use
the Derksen ideal'', since this ideal turns up in different contexts.

The Derksen ideal is defined as the ideal of polynomials in
$K[\mathbf x,\mathbf y]$ vanishing on all
$\bigl(v,\sigma(v)\bigr) \in V \times V$ with $\sigma \in G$,
$v \in V = K^n$, so
\[
  D = \bigcap_{\sigma \in G} \bigl(y_1 - \sigma(x_1) \upto y_n -
  \sigma(x_n)\bigr).
\]
This formula is better suited for generalizations than the original
geometric definition of $D$. For example, if a group $G$ acts on a
finitely generated (but not necessarily finite) field extension
$L = K(a_1 \upto a_n)$ of $K$, we
define the Derksen ideal as
\[
  D := \bigcap_{\sigma \in G} \bigl(y_1 - \sigma(a_1) \upto y_n -
  \sigma(a_n)\bigr) \subseteq L[y_1 \upto y_n].
\]
Of course this depends on the choice of the generators~$a_i$. An
important special case is that the $a_i$ are algebraically
independent, so $L$ is a rational function field. If $G$ is infinite,
this definition does not provide a way to compute $D$. So let us
assume that $G$, as in the previous section, is a linear algebraic
group given as a closed subset of $K^r$ by its vanishing ideal
$\Id(G) = (g_1 \upto g_l) \subseteq K[z_1 \upto z_r]$. Moreover,
assume that the $G$-action is by $K$-automorphisms and given by
polynomials $f_1 \upto f_n \in L[\mathbf z]$ such that
\begin{equation} \label{eqMorphic}%
  \sigma(a_i) = f_i(\sigma)
\end{equation}
(where~$f_i(\sigma)$ means specializing the $z_j$ to the coordinates
of~$\sigma \in K^r$). These are reasonable assumptions, and they are
good enough to make the Derksen ideal computable. In fact, we have
virtually the same result as in the last section:

\begin{prop}[Computing the Derksen ideal] \label{pDerksen}%
  In the above situation set
  \[
    \widehat D := \bigl(g_1 \upto g_l,f_1 - y_1 \upto f_n - y_n\bigr)
    \subseteq L[\mathbf y,\mathbf z].
  \]
  Then $D = L[\mathbf y] \cap \widehat D$.
\end{prop}

\begin{proof}
  We start by taking $f \in D$. Set
  $\widetilde f := f(f_1 \upto f_n) \in L[\mathbf z]$. Then
  $f - \widetilde f \in \widehat{D}$, so for the inclusion
  ``$\subseteq$'' we need to show $\widetilde f \in \widehat D$. Let
  $\sigma \in G \subseteq K^r$. Then
  \[
    \widetilde f(\sigma) = f\bigl(f_1(\sigma) \upto f_n(\sigma)\bigr)
    = f\bigl(\sigma(a_1) \upto \sigma(a_n)\bigr) = 0,
  \]
  where the last equality comes from
  $f \in \bigl(y_1 - \sigma(a_1) \upto y_n - \sigma(a_n)\bigr)$. Let
  $B$ be a basis of $L$ as a vector space over $K$. Then we can write
  $\widetilde f = \sum_{b \in B} h_b \cdot b$ with
  $h_b \in K[\mathbf z]$, so
  $\sum_{b \in B} h_b(\sigma) \cdot b = \widetilde{f}(\sigma) = 0$,
  which implies $h_b(\sigma) = 0$ for all~$b$. Since this holds for
  every $\sigma \in G$, we conclude $h_b \in (g_1 \upto g_l)$ (as an
  ideal in $K[\mathbf z]$), so $\widetilde f \in (g_1 \upto g_l)$ (as
  an ideal in $L[\mathbf z]$). This implies
  $\widetilde f \in \widehat{D}$.

  For the reverse inclusion, take
  $f \in L[\mathbf y ] \cap \widehat{D}$. Then
  $f = \sum h_i g_i + \sum h_i' (f_i - y_i)$ with
  $h_i,h_i' \in L[\mathbf y,\mathbf z]$. Let $\sigma \in
  G$. Viewing~$f$ as a polynomial in $L[\mathbf y,\mathbf z]$ we can
  specialize the $z_i$ to the coordinates of~$\sigma$, which does not
  change~$f$. So
  \begin{multline*}
    f = f(\sigma) = \sum h_i(\sigma) g_i(\sigma) + \sum h_i'(\sigma)
    \bigl(f_i(\sigma) - y_i\bigr) = \\
    = \sum h_i'(\sigma) \bigl(\sigma(a_i) - y_i\bigr) \in \bigl(y_1 -
    \sigma(a_1) \upto y_n - \sigma(a_n)\bigr).
  \end{multline*}
  Since this holds for every $\sigma \in G$, we obtain $f \in D$.
\end{proof}

The following result is amazing because it uses the Derksen ideal in a
way that is very different from \cref{tDerksen}, but still arrives at
computing generating invariants. It is also remarkable that there is
no hypothesis such as reductivity required for the group. The theorem
probably goes back to M{\"u}ller-Quade and Beth~\cite{MQB:99}, but
their result was modified, simplified, generalized and extended by
quite a few authors, see \cite{HK:06,Kamke:Kemper:2011,kemper:2015}.

\begin{theorem}[Computing invariant fields] \label{tField}%
  In the above situation, let $\mathcal G \subseteq L[\mathbf y]$ be a
  reduced Gr\"obner basis, with respect to any monomial order, of the
  Derksen ideal. Then the invariant field $L^G$ is generated, as an
  extension of $K$, by the coefficients of all polynomials in
  $\mathcal G$.
\end{theorem}

\begin{proof}
  With $G$ acting coefficient-wise on $L[\mathbf y]$, the action
  preserves the set of monomials in a polynomial. So for
  $\sigma \in G$, $\sigma(\mathcal G)$ is a reduced Gr\"obner basis of
  $\sigma(D) = D$. The uniqueness of reduced Gr\"obner bases (see
  \cite[Theorem~5.43]{BW}) yields $\sigma(\mathcal G) = \mathcal G$,
  so~$\sigma$ fixes every polynomial in $\mathcal G$. Therefore the
  field $L'$ generated by the coefficients of all polynomials in
  $\mathcal G$ is contained in $L^G$.

  For the reverse inclusion, let $b \in L^G$, which we can write as
  $b = \frac{f(a_1 \upto a_n)}{g(a_1 \upto a_n)}$ with
  $f,g \in K[\mathbf y]$. Setting $h := f - b g \in L[\mathbf y]$ and
  taking $\sigma \in G$, we have
  \[
    h\bigl(\sigma(a_1) \upto \sigma(a_n)\bigr) = \sigma\bigl(f(a_1
    \upto a_n) - b g(a_1 \upto a_n)\bigr) = 0,
  \]
  so $h \in D$ and therefore the normal form $\NF_{\mathcal G}(h)$ is
  zero. So $\NF_{\mathcal G}(f) - b \NF_{\mathcal G}(g) = 0$ by the
  $L$-linearity of the normal form. Since $g(a_1 \upto a_n) \ne 0$, we
  have $g \notin D$, so $\NF_{\mathcal G}(g) \ne 0$. This gives
  $b = \NF_{\mathcal G}(f)/\NF_{\mathcal G}(g)$. But computing the
  normal forms of~$f$ and~$g$ only involves polynomials from
  $L'[\mathbf y]$, so $b \in L'(\mathbf y) \cap L = L'$.
\end{proof}

\begin{ex}
  As a toy example, consider the action of the multiplicative group
  $G = \mathbb G_m$ on two variables~$x_1$ and~$x_2$ with weight
  $(1,-1)$. With the notation of \cref{pDerksen} we have
  \begin{multline*}
    \widehat D = \bigl(z_1 z_2 - 1, z_1 x_1 -y_1, z_2 x_2 - y_2\bigr)
    = \left(z_1 z_2 - 1, z_1 - \frac{1}{x_1} y_1, z_2 - \frac{1}{x_2}
      y_2\right) \\
    = \left(y_1 y_2 - x_1 x_2, z_1 - \frac{1}{x_1} y_1, z_2 -
      \frac{1}{x_2} y_2\right),
  \end{multline*}
  where the last displayed generating set is the reduced Gr\"obner
  basis w.r.t.\ any monomial order with $z_i > y_j$. So
  $y_1 y_2 - x_1 x_2$ is the reduced Gr\"obner basis of $D$, and
  $K(x_1,x_2)^G = K(x_1 \cdot x_2)$. Of course, we have known this all
  along.
\end{ex}

As mentioned above, four papers were cited for a result as simple as
\cref{tField} because it can be generalized and extended in various
ways. For example, for a group action on an irreducible affine variety
$X$, it is often possible to find a $G$-invariant common
denominator~$f$ for the coefficients of $\mathcal G$, and then one
gets a generating set of the localization $K[X]^G_f$. This does not
require $G$ to be reductive, or even $K[X]^G$ to be finitely
generated. In fact, there is a semi-algorithm for computing $K[X]^G$
from $K[X]^G_f$, which terminates if and only if the desired invariant
ring is finitely generated. (But, alas, there is no algorithm known
for determining finite generation.) \\

A further way of using the Derksen ideal forms a bridge to the next
section about separating invariants. Let $G$ be a reductive group in
positive characteristic, so it is not linearly reductive unless it is
a torus with a nonmodular finite group on top of it. Then, as for any
reductive group, we know that all invariants agree on two points
$v,w \in V$ if and only if the orbit closures intersect:
\begin{equation} \label{eqSeparate}%
  f(v) = f (w) \quad \text{for all} \quad f \in \R^G \quad
  \Longleftrightarrow \quad \overline{G(v)} \cap \overline{G(w)} \ne
  \emptyset.
\end{equation}
It is not hard to derive from this how to calculate the \df{separating
  variety} $\mathcal S$, defined to contain all pairs of points
$(v,w) \in V \times V$ where all invariants agree. In fact,
$\mathcal S$ is given by the elimination ideal
$K[\mathbf x,\mathbf y] \cap \bigl(D_{\mathbf x,\mathbf z},D_{\mathbf
  y,\mathbf z}\bigl)$, where $D_{\mathbf x,\mathbf z}$ and
$D_{\mathbf y,\mathbf z}$ stand for the Derksen ideal, as first
defined in the previous section, with the~$y_i$ or~$x_i$,
respectively, replaced by new variables~$z_i$. With this, one can also
compute separating invariants: just add homogeneous invariants~$f_i$
of rising degrees until the $f_i(\mathbf x) - f_i(\mathbf y)$ define
the variety $\mathcal S$. So we can compute a separating homogeneous
subalgebra $A \subseteq \R^G$. Now it is well known that (as we are in
characteristic~$p > 0$) the entire invariant ring is the purely
inseparable closure of $A$, i.e.,
$\R^G = \bigl\{f \in \R \mid f^q \in A \ \text{for some}\
p\text{-power} \ q\bigr\}$. And for the calculation of the purely
inseparable closure there is also an algorithm (using Gr\"obner basis
methods, as readers are probably expecting). In summary, we obtain an
algorithm for computing invariant rings of reductive groups in
positive characteristic, which complements Derksen's algorithm for
reductive groups in characteristic~$0$. For more details, see
\cite{kem.separating}, \cite[Section~4.9]{Derksen:Kemper:2015}, and
\cite{Derksen.Kemper06}, where the methods are extended to reductive
group actions on affine varieties instead of only vector spaces.

\section{Separating invariants of infinite groups}
\label{sSepInf}%

In the previous section we have considered separating invariants of
reductive group, so the focus here is on nonreductive groups. For such
groups, the invariant ring may not be finitely generated, and the
criterion~\cref{eqSeparate} may fail. Nevertheless, there always
exists a finite set of separating invariants. This may seem like an
amazing, deep fact, until one realizes how simple it is to prove: the
ideal in $K[\mathbf x,\mathbf y]$ generated by all
$\Delta f := f(\mathbf x) - f(\mathbf y)$ with $f \in \R^G$ can, by
Hilbert's basis theorem, be generated by the deltas of finitely many
invariants, which then form a separating set.

This argument is similar in spirit to Hilbert's first proof of finite
generation of $\R^G$ for linearly reductive groups. After Hilbert's
proof, it took about 100 years until a constructive version emerged in
the form of Derksen's algorithm. For separating invariants, we still
seem to be in the 100-year waiting period: at the time of writing
this, no algorithm has been found for computing a finite separating
set in the case of nonreductive groups.

We can, however, go further in another direction and quantify our
finiteness statement. The following result gives an upper bound on the
minimal number of separating invariants that is, surprisingly,
independent of the group. It is hard to pin down who first proved the
result, which has been folklore for quite a while.

\begin{theorem}[the number of separating
  invariants] \label{tSeparating}%
  If $K$ is an infinite field, there is a set of separating invariants
  of size $\le 2 n + 1$, with~$n$ the number of variables.
\end{theorem}

Since the proof is nice and short, we present it here.

\begin{proof}
  We know from the above observation that there is a finite set of
  separating invariants $f_1 \upto f_k$. Assume $k > 2 n + 1$. Then
  with~$t$ an additional variable, the
  \[
    g_i := t \cdot \bigl(f_i(\mathbf x) - f_i(\mathbf y)\bigr) \in
    K[\mathbf x,\mathbf y,t] \quad (i = 1 \upto k)
  \]
  are algebraically dependent, so we have a nonzero polynomial $H$
  in~$k$ variables such that $H(g_1 \upto g_k) = 0$. We can choose
  $\alpha_1 \upto \alpha_k \in K$ with $H(\alpha_1 \upto \alpha_k) \ne
  0 $ and $\alpha_1 \ne 0$. Now we claim that the invariants
  \[
    \widetilde f_i := \alpha_1 f_i - \alpha_i f_1 \quad (i = 2 \upto
    k)
  \]
  form a separating set. If that is proved, we can repeat this until
  reaching a separating set of size $2 n + 1$. To prove the claim,
  assume that the $\widetilde f_i$ do not form a separating set, so
  there are points $v,w \in K^n$ such that all $\widetilde f_i$ agree
  on~$v$ and~$w$, but not all~$f_i$. Then from the definition of the
  $\widetilde f_i$ we see that
  \[
    \alpha_1\bigl(f_i(v) - f_i(w)\bigr) = \alpha_i \bigl(f_1(v) -
    f_1(w)\bigr),
  \]
  so $f_1(v) \ne f_1(w)$. Now let $\Phi$: $K[\mathbf x,\mathbf y,t]
  \to K$ be the map given by evaluating a polynomial at the point
  $(v,w,\eta)$ with $\eta := \frac{\alpha_1}{f_1(v) - f_1(w)}$. Then
  \[
    \Phi(g_i) = \eta \cdot \bigl(f_i(v) - f_i(w)\bigr) = \alpha_i
    \quad (i = 1 \upto k),
  \]
  from which the contradiction
  \[
    H(\alpha_1 \upto \alpha_k) = H\bigl(\Phi(g_1) \upto
    \Phi(g_k)\bigr) = \Phi\bigl(H(g_1 \upto g_k)\bigr) = \Phi(0) = 0
  \]
  ensues.
\end{proof}

\begin{rem}
  The above proof really establishes the upper bound
  $2 \dim\bigl(\R^G\bigr) + 1$, because for the invariant ring, as for
  any subalgebra of a finitely generated algebra, the Krull dimension
  is equal to the transcendence degree (see
  \cite[Exercise~5.3]{Kemper.Comalg} or
  \cite[Proposition~2.3]{Giral:81}). Also notice that the proof never
  uses the group action, so the theorem holds in a broader context
  (see \cite[Theorem~5.3]{Kamke:Kemper:2011}). The proof is
  constructive, so once a finite separating set is known, it can be
  boiled down to one of size $\le 2 n + 1$. But a disadvantage is that
  starting out with a homogeneous separating set will produce a
  smaller separating set that is almost certainly inhomogeneous, since
  the smaller set consists of linear combinations of the larger set.
\end{rem}

Even though the focus in this section lies on nonreductive groups,
\cref{tSeparating} tells us something new also in the case of
reductive and even finite groups. In fact, there are lots of examples
where the minimum number of {\em generating} invariants vastly exceeds
the bound $2 n + 1$ (see \cite[Section~5]{Kamke:Kemper:2011}).


\begin{bibdiv}
\begin{biblist}

\bib{AF}{inproceedings}{
      author={Almkvist, Gert},
      author={Fossum, Robert~M.},
       title={Decompositions of exterior and symmetric powers of indecomposable
  {${\mathbb Z}/p{\mathbb Z}$}-modules in characteristic~$p$ and relations to
  invariants},
        date={1976--1977},
   booktitle={S{\'e}m. d'alg{\`e}bre {P}. {D}ubreil},
      series={Lecture Notes in Math.},
   publisher={Springer-Verlag},
     address={Berlin, Heidelberg, New York},
       pages={1\ndash 111},
}

\bib{BW}{book}{
      author={Becker, Thomas},
      author={Weispfenning, Volker},
       title={{G}r{\"o}bner bases},
   publisher={Springer-Verlag},
     address={Berlin, Heidelberg, New York},
        date={1993},
}

\bib{bens}{book}{
      author={Benson, David~J.},
       title={Polynomial invariants of finite groups},
      series={Lond.\ Math.\ Soc.\ Lecture Note Ser.},
   publisher={Cambridge Univ.\ Press},
     address={Cambridge},
        date={1993},
      number={190},
}

\bib{BSM:2018}{article}{
      author={Blum-Smith, Ben},
      author={Marques, Sophie},
       title={When are permutation invariants {C}ohen-{M}acaulay over all
  fields?},
        date={2018},
     journal={Algebra Number Theory},
      volume={12},
       pages={1787\ndash 1821},
}

\bib{magma}{article}{
      author={Bosma, Wieb},
      author={Cannon, John~J.},
      author={Playoust, Catherine},
       title={The {M}agma algebra system~{I}: The user language},
        date={1997},
     journal={J.~Symb. Comput.},
      volume={24},
       pages={235\ndash 265},
}

\bib{bou:81}{book}{
      author={Bourbaki, Nicolas},
       title={Groupes et alg\`{e}bres de {L}ie, chap.\ 4, 5, et 6},
   publisher={Masson},
     address={Paris},
        date={1981},
}

\bib{CGHSW:b}{article}{
      author={Campbell, H. E.~A.},
      author={Geramita, A.~V.},
      author={Hughes, I.~P.},
      author={Shank, R.~J.},
      author={Wehlau, D.~L.},
       title={Non-{C}ohen-{M}acaulay vector invariants and a {N}oether bound
  for a {G}orenstein ring of invariants},
        date={1999},
     journal={Canad.\ Math.\ Bull.},
      volume={42},
       pages={155\ndash 161},
}

\bib{Chevalley}{article}{
      author={Chevalley, Claude},
       title={Invariants of finite groups generated by reflections},
        date={1955},
     journal={Amer. J. Math.},
      volume={77},
       pages={778\ndash 782},
}

\bib{Derksen:99}{article}{
      author={Derksen, Harm},
       title={Computation of invariants for reductive groups},
        date={1999},
     journal={Adv. Math.},
      volume={141},
       pages={366\ndash 384},
}

\bib{Derksen:Kemper}{book}{
      author={Derksen, Harm},
      author={Kemper, Gregor},
       title={Computational invariant theory},
      series={Encyclopaedia of Mathematical Sciences},
   publisher={Springer-Verlag},
     address={Berlin, Heidelberg, New York},
        date={2002},
      number={130},
}

\bib{Derksen.Kemper06}{article}{
      author={Derksen, Harm},
      author={Kemper, Gregor},
       title={Computing invariants of algebraic group actions in arbitrary
  characteristic},
        date={2008},
     journal={Adv. Math.},
      volume={217},
       pages={2089\ndash 2129},
}

\bib{Derksen:Kemper:2015}{book}{
      author={Derksen, Harm},
      author={Kemper, Gregor},
       title={Computational invariant theory},
     edition={2},
      series={Encyclopaedia of Mathematical Sciences},
   publisher={Springer, Heidelberg},
     address={Heidelberg, New York, Dordrecht, London},
        date={2015},
      number={130},
}

\bib{Dufresne:2009}{article}{
      author={Dufresne, Emilie},
       title={Separating invariants and finite reflection groups},
        date={2009},
     journal={Adv. Math.},
      volume={221},
      number={6},
       pages={1979\ndash 1989},
}

\bib{Dufresne:Jeffries:2013}{article}{
      author={Dufresne, Emilie},
      author={Jeffries, Jack},
       title={Separating invariants and local cohomology},
        date={2015},
     journal={Adv. Math.},
      volume={270},
       pages={565\ndash 581},
}

\bib{eis}{book}{
      author={Eisenbud, David},
       title={Commutative algebra with a view toward algebraic geometry},
   publisher={Springer-Verlag},
     address={New York},
        date={1995},
}

\bib{ES}{article}{
      author={Ellingsrud, Geir},
      author={Skjelbred, Tor},
       title={Profondeur d'anneaux d'invariants en caract{\'e}ristique $p$},
        date={1980},
     journal={Compos.\ Math.},
      volume={41},
       pages={233\ndash 244},
}

\bib{FGGHMN:2020}{unpublished}{
      author={Ferraro, Luigi},
      author={Galetto, Federico},
      author={Gandini, Francesca},
      author={Huang, Hang},
      author={Mastroeni, Matthew},
      author={Ni, Xianglong},
       title={The {I}nvariant{R}ing package for {M}acaulay2},
        date={2020},
        note={arXiv preprint, see \url{https://arxiv.org/abs/2010.15331}},
}

\bib{fleisch:c}{article}{
      author={Fleischmann, Peter},
       title={The {N}oether bound in invariant theory of finite groups},
        date={2000},
     journal={Adv. in Math.},
      volume={156},
       pages={23\ndash 32},
}

\bib{Fogarty:2001}{article}{
      author={Fogarty, John},
       title={On {N}oether's bound for polynomial invariants of a finite
  group},
        date={2001},
     journal={Electron. Res. Announc. Amer. Math. Soc.},
      volume={7},
       pages={5\ndash 7},
}

\bib{Giral:81}{article}{
      author={Giral, Jos{\'e}~M.},
       title={Krull dimension, transcendence degree and subalgebras of finitely
  generated algebras},
        date={1981},
     journal={Arch. Math. (Basel)},
      volume={36},
       pages={305\ndash 312},
}

\bib{Gordeev.Kemper}{article}{
      author={Gordeev, Nikolai},
      author={Kemper, Gregor},
       title={On the branch locus of quotients by finite groups and the depth
  of the algebra of invariants},
        date={2003},
     journal={J. Algebra},
      volume={268},
       pages={22\ndash 38},
}

\bib{Hartshorne:1962}{article}{
      author={Hartshorne, Robin},
       title={Complete intersections and connectedness},
        date={1962},
     journal={Amer. J. Math.},
      volume={84},
       pages={497\ndash 508},
}

\bib{HK:06}{article}{
      author={Hubert, Evelyne},
      author={Kogan, Irina~A.},
       title={Rational invariants of an algebraic groups action. {C}onstructing
  and rewriting},
        date={2007},
     journal={J. Symb. Comput.},
      volume={42},
       pages={203\ndash 217},
}

\bib{Hughes:Kemper:b}{article}{
      author={Hughes, Ian},
      author={Kemper, Gregor},
       title={Symmetric powers of modular representations for groups with a
  {S}ylow subgroup of prime order},
        date={2001},
     journal={J. of Algebra},
      volume={241},
       pages={759\ndash 788},
}

\bib{hup}{book}{
      author={Huppert, Bertram},
       title={{E}ndliche {G}ruppen~{I}},
   publisher={Springer-Verlag},
     address={Berlin, Heidelberg, New York},
        date={1967},
}

\bib{Kamke:Kemper:2011}{article}{
      author={Kamke, Tobias},
      author={Kemper, Gregor},
       title={Algorithmic invariant theory of nonreductive groups},
        date={2012},
     journal={Qualitative Theory of Dynamical Systems},
      volume={11},
       pages={79\ndash 110},
}

\bib{KS:02}{article}{
      author={Karagueuzian, Dikran~B.},
      author={Symonds, Peter},
       title={The module structure of a group action on a polynomial ring: a
  finiteness theorem},
        date={2007},
     journal={J. Amer. Math. Soc.},
      volume={20},
       pages={931\ndash 967},
}

\bib{kem:g}{article}{
      author={Kemper, Gregor},
       title={An algorithm to calculate optimal homogeneous systems of
  parameters},
        date={1999},
     journal={J.~Symb. Comput.},
      volume={27},
       pages={171\ndash 184},
}

\bib{kem:h}{article}{
      author={Kemper, Gregor},
       title={On the {C}ohen-{M}acaulay property of modular invariant rings},
        date={1999},
     journal={J. of Algebra},
      volume={215},
       pages={330\ndash 351},
}

\bib{kem:flat}{article}{
      author={Kemper, Gregor},
       title={The depth of invariant rings and cohomology, {{\normalfont with
  an appendix by Kay Magaard}}},
        date={2001},
     journal={J. of Algebra},
      volume={245},
       pages={463\ndash 531},
}

\bib{kem.separating}{article}{
      author={Kemper, Gregor},
       title={Computing invariants of reductive groups in positive
  characteristic},
        date={2003},
     journal={Transformation Groups},
      volume={8},
       pages={159\ndash 176},
}

\bib{Kemper.Comalg}{book}{
      author={Kemper, Gregor},
       title={A course in commutative algebra},
      series={Graduate Texts in Mathematics},
   publisher={Springer-Verlag},
     address={Berlin, Heidelberg},
        date={2011},
      number={256},
}

\bib{kemper:2015}{article}{
      author={Kemper, Gregor},
       title={Using extended {D}erksen ideals in computational invariant
  theory},
        date={2016},
     journal={J. Symbolic Comput.},
      volume={72},
       pages={161\ndash 181},
}

\bib{Kemper:Lopatin:Reimers:2022}{article}{
      author={Kemper, Gregor},
      author={Lopatin, Artem},
      author={Reimers, Fabian},
       title={Separating invariants over finite fields},
        date={2022},
     journal={Journal of Pure and Applied Algebra},
      volume={226},
}

\bib{King:2013}{article}{
      author={King, Simon~A.},
       title={Minimal generating sets of non-modular invariant rings of finite
  groups},
        date={2013},
     journal={J. Symbolic Comput.},
      volume={48},
       pages={101\ndash 109},
}

\bib{Kreuzer:Robbiano:2005}{book}{
      author={Kreuzer, Martin},
      author={Robbiano, Lorenzo},
       title={Computational commutative algebra 2},
   publisher={Springer, Berlin, Heidelberg},
        date={2005},
}

\bib{Lorenz.Pathak}{article}{
      author={Lorenz, Martin},
      author={Pathak, Jay},
       title={On {C}ohen-{M}acaulay rings of invariants},
        date={2001},
     journal={J. of Algebra},
      volume={245},
       pages={247\ndash 264},
}

\bib{MQB:99}{inproceedings}{
      author={M{\"u}ller-Quade, J{\"o}rn},
      author={Beth, Thomas},
       title={Calculating generators for invariant fields of linear algebraic
  groups},
        date={1999},
   booktitle={Applied algebra, algebraic algorithms and error-correcting codes
  ({H}onolulu, {HI})},
      series={Lecture Notes in Comput. Sci.},
   publisher={Springer},
     address={Berlin},
       pages={392\ndash 403},
}

\bib{noe:a}{article}{
      author={Noether, Emmy},
       title={{D}er {E}ndlichkeitssatz der {I}nvarianten endlicher {G}ruppen},
        date={1916},
     journal={Math.\ Ann.},
      volume={77},
       pages={89\ndash 92},
}

\bib{rich:b}{article}{
      author={Richman, David~R.},
       title={Invariants of finite groups over fields of characteristic~$p$},
        date={1996},
     journal={Adv. in Math.},
      volume={124},
       pages={25\ndash 48},
}

\bib{Serre68}{inproceedings}{
      author={Serre, Jean-Pierre},
       title={Groupes finis d'automorphismes d'anneaux locaux r{\'e}guliers},
        date={1968},
   booktitle={Colloque d'alg{\`e}bre},
   publisher={Secr{\'e}tariat math{\'e}matique},
     address={Paris},
       pages={8\ndash 01 \ndash  8\ndash 11},
}

\bib{lsm}{book}{
      author={Smith, Larry},
       title={Polynomial invariants of finite groups},
   publisher={A.~K.~Peters},
     address={Wellesley, Mass.},
        date={1995},
}

\bib{stu}{book}{
      author={Sturmfels, Bernd},
       title={Algorithms in invariant theory},
   publisher={Springer-Verlag},
     address={Wien, New York},
        date={1993},
}

\bib{Symonds:2011}{article}{
      author={Symonds, Peter},
       title={On the {C}astelnuovo-{M}umford regularity of rings of polynomial
  invariants},
        date={2011},
     journal={Ann. of Math. (2)},
      volume={174},
       pages={499\ndash 517},
}

\end{biblist}
\end{bibdiv}
\end{document}